\documentclass[reqno]{amsart}
\usepackage{graphicx}
\usepackage{caption}
\usepackage{subcaption}
\usepackage{color}
\usepackage{pdfpages}
\usepackage{svg}
\usepackage{tikz}

\usepackage{amssymb}
\usepackage{mathrsfs}
\usepackage{enumitem}
\usepackage{hyperref}
\usepackage{isomath}
\usepackage[toc,page]{appendix}
\usepackage{multicol}
\usepackage{seqsplit}

\begin{document}

\newtheorem{thm}{Theorem}[section]
\newtheorem*{thm1}{Theorem}
\newtheorem{lem}[thm]{Lemma}
\newtheorem{cor}[thm]{Corollary}
\newtheorem{add}[thm]{Addendum}
\newtheorem{prop}[thm]{Proposition}
\theoremstyle{definition}
\newtheorem{defn}[thm]{Definition}
\newtheorem{claim}[thm]{Claim}
\newtheorem*{mainthm}{Theorem 1.1}
\theoremstyle{remark}
\newtheorem{rmk}[thm]{Remark}
\newtheorem{ex}[thm]{Example}
\newtheorem{conj}[thm]{Conjecture}

\newtheorem{algo}[thm]{Algorithm}

\newcommand{\CC}{\mathbb{C}}
\newcommand{\RR}{\mathbb{R}}
\newcommand{\DD}{\mathbb{D}}

\newcommand{\ZZ}{\mathbb{Z}}
\newcommand{\QQ}{\mathbb{Q}}
\newcommand{\NN}{\mathbb{N}}
\newcommand{\FF}{\mathbb{F}}
\newcommand{\PP}{\mathbb{P}}
\newcommand{\CmodTwoPiIZ}{{\mathbf C}/2\pi i {\mathbf Z}}
\newcommand{\Cnozero}{{\mathbf C}\backslash \{0\}}
\newcommand{\Cinfty}{{\mathbf C}_{\infty}}
\newcommand{\RPnminustwo}{\mathbb{RP}^{n-2}}

\newcommand{\SLtwoC}{\mathrm{SL}(2,{\mathbb C})}
\newcommand{\SLtwoR}{\mathrm{SL}(2,{\mathbb R})}
\newcommand{\PSLtwoC}{\mathrm{PSL}(2,{\mathbb C})}
\newcommand{\PSLtwoR}{\mathrm{PSL}(2,{\mathbb R})}
\newcommand{\SLtwoZ}{\mathrm{SL}(2,{\mathbb Z})}
\newcommand{\PSLtwoZ}{\mathrm{PSL}(2,{\mathbb Z})}
\newcommand{\PGLtwoQ}{\mathrm{PGL}(2,{\mathbb Q})}

\newcommand{\A}{{\mathcal A}}
\newcommand{\B}{{\mathcal B}}
\newcommand{\C}{{\mathcal X}}
\newcommand{\D}{{\mathcal D}}
\newcommand{\E}{{\mathcal E}}
\newcommand{\F}{{\mathcal F}}

\newcommand{\OO}{{\mathcal O}}

\newcommand{\MCG}{\mathcal{MCG}}
\newcommand{\EE}{\mathbb{E}^2}
\newcommand{\HH}{\mathbb{H}^2}
\newcommand{\HHH}{\mathbb{H}^3}
\newcommand{\TTT}{\mathcal{T}}
\newcommand{\QQQ}{\mathcal{Q}}
\newcommand{\SSS}{\mathcal{S}}
\newcommand{\tr}{{\hbox{tr}\,}}
\newcommand{\Hom}{\mathrm{Hom}}
\newcommand{\Aut}{\mathrm{Aut}}
\newcommand{\Inn}{\mathrm{Inn}}
\newcommand{\Out}{\mathrm{Out}}
\newcommand{\SL}{\mathrm{SL}}
\newcommand{\BQ}{\rm{BQ}}
\newcommand{\Id}{\rm{Id}}
\newcommand{\Width}{\rm{Width}}

\newcommand{\setn}{{[n]}}
\newcommand{\powern}{{P(n)}}
\newcommand{\nck}{{C(n,k)}}

\newcommand{\hatI}{{\hat{I}}}
\newcommand{\TkDelta}{{T^{|k|}(\Delta)}}
\newcommand{\vecDelta}{{\vec{\Delta}_{\phi}}}

\newcommand{\Tabstwo}{{T^{|2|}(\Delta)}}
\newcommand{\Tnminusone}{{T^{|n-1|}(\Delta)}}
\newcommand{\Hur}{{\mathcal{H}}}
\newcommand{\JW}{{\hbox{JW}}}
\newcommand{\wt}{{\hbox{wt}}}
\newcommand{\Deltaone}{{\Delta^{(1)}}}
\newcommand{\Deltatwo}{{\Delta^{(2)}}}
\newcommand{\Deltathree}{{\Delta^{(3)}}}
\newcommand{\Deltan}{{\Delta^{(n)}}}

\newenvironment{pf}{\noindent {\it Proof.}\quad}{\square \vskip 10pt}

\title[Pseudomodular groups]{Hyperbolic jigsaws and families  of pseudomodular groups II}
\author[B. Lou, S.P. Tan and A.D. Vo ]{Beicheng Lou,  Ser Peow Tan and  Anh Duc Vo }

\address{Department of Applied Physics \\
Stanford University \\
USA} \email{lbc45123@hotmail.com}
\address{Department of Mathematics \\
	National University of Singapore \\
	Singapore 119076} \email{mattansp@nus.edu.sg}
\address{Department of Mathematics \\Harvard University \\
	USA} \email{ducvo@math.harvard.edu}

\subjclass[2000]{}


\date{\today}

%
%

 \begin{abstract}
	In our previous paper \cite{LTV}, we introduced a hyperbolic jigsaw construction and constructed infinitely many non-commensurable,  non-uniform, non-arithmetic lattices of $\PSLtwoR$ with cusp set $\QQ \cup \{\infty\}$ (called pseudomodular groups by Long and Reid in \cite{LR}), thus answering a question posed by Long and Reid.  In this paper, we continue with our study of these  jigsaw groups exploring questions of arithmeticity, pseudomodularity, and also related pseudo-euclidean and continued fraction algorithms arising from these groups. We also answer another question of Long and Reid in \cite{LR} by demonstrating a recursive formula for the tessellation of the hyperbolic plane arising from Weierstrass groups which generalizes the well-known ``Farey addition'' used to generate the Farey tessellation.

\end{abstract}

%

 \maketitle
 \tableofcontents

 \vspace{10pt}
\section{Introduction} In \cite{LR}, Long and Reid showed the existence of four non-commensurable {\it pseudo-modular} surfaces/groups, using punctured torus groups. These are {\em non-arithmetic} surfaces/lattices which share the same  cusp set as the modular surface $\HH/\PSLtwoZ$, that is, the cusp set is all of $\QQ \cup \{\infty\}$. In \cite{LTV}, we introduced a construction of ideal polygons ({\it hyperbolic jigsaws}) by glueing ideal triangles with marked points on the sides ({\it jigsaw tiles}) along matching sides. Such an ideal polygon ({\it jigsaw}) $J$ is the fundamental domain for the subgroup $\Gamma_J$ of $\PSLtwoR$ generated by the $\pi$-rotations about the marked points on the sides of $J$. The resulting surface is complete provided that the tiles satisfy a balancing condition. By using sets of tiles ({\it Jigsaw sets}) of two particular types ($\SSS(1,2):=\{\triangle(1,1,1), \triangle(1,1/2,2) \}$ and $\SSS(1,3):=\{\triangle(1,1,1), \triangle(1,1/3,3)\}$), we were able to construct infinitely many non-commensurable pseudomodular surfaces, using either of these two jigsaw sets. 

In this follow-up paper, we continue our study of the jigsaw construction, and further explore  the questions of non-arithmeticity and pseudomodularity of the lattices constructed from more general jigsaw sets, as well as the  pseudo-euclidean algorithms and generalized continued fractions arising from these lattices with a view of developing a more general theory.  We also answer a question raised in \cite{LR} concerning the recursive formula for the tessellation of $\HH$ arising from punctured torus groups. As far as possible, we will adopt the notation of \cite{LTV} to avoid confusion. We will recall these as necessary in the subsequent sections. 

We first start with the question of Long and Reid. A punctured torus group $\Gamma<\PSLtwoR$ together with a ``normalized'' ideal quadrilateral fundamental domain $Q$ gives rise to a tessellation  $\QQQ$ of the hyperbolic plane by ideal quadrilaterals which are translates of $Q$ by $\Gamma$. Long and Reid asked for a recursive formula for $\QQQ$, similar to that for the Farey tessellation of the plane. By considering Weierstrass groups which are extensions of the punctured torus groups, we obtain triangulations $\TTT$ of $\HH$  which can be described recursively from the parameters. The recursion for the tessellation by quadrilaterals can then be easily deduced. Recall that the Weierstrass group $\Gamma(k_1,k_2,k_3)$ where $k_1k_2k_3=1$ is the Coxeter group generated by
\begin{equation*}
\Small
\iota_1 = \frac{1}{\sqrt{k_1}}\left(
\begin{array}{cc}
k_1 & 1+k_1 \\
-k_1 & -k_1 \\
\end{array}
\right), \enskip
\iota_2 =
\frac{1}{\sqrt{k_2}} \left(
\begin{array}{cc}
1 & 1 \\
-(k_2+1) & -1 \\
\end{array}
\right), \enskip
\iota_3 =
\frac{1}{\sqrt{k_3}} \left(
\begin{array}{cc}
0 & k_3 \\
-1 & 0 \\
\end{array}
\right).\end{equation*}

\medskip

\begin{thm} {\rm{(Theorem \ref{thm:recursion})}}
Let $\Gamma:=\Gamma(k_1,k_2,k_3)$ be a Weierstrass group and $\TTT$ be the triangulation of $\HH$ obtained by taking the translates of the fundamental triangle $\triangle_0:=(\infty, -1, 0)$ by $\Gamma$. Let $(a,b,c)=g\cdot \triangle_0$ be a level $n$ triangle,  where $g\in \Gamma=\langle ~ \iota_1,\iota_2,\iota_3  ~|~\iota_1^2=\iota_2^2=\iota_3^2=\hbox{Id}~\rangle$ is a word with reduced word length $n$, and $a=g(\infty)$, $b=g(-1)$, $c=g(0)$. 
Let $T_1$, $T_2$ and $T_3$ be the triangles of $\TTT$ adjacent to $(a,b,c)$, where $T_i$ shares the side marked by $k_i$ with $(a,b,c)$. Then
\begin{equation}\label{eqn:recursion}
\begin{split}
    T_{1} & = \left(b,~a,~~\frac{k_{1}b(c-a)+a(c-b)}{k_{1}(c-a)+(c-b)}\right)
     \\
     T_{2} &=\left(\frac{k_{2}c(a-b)+b(a-c)}{k_{2}(a-b)+(a-c)},~~c,~b\right)
    \\
    T_{3} &= \left(c,~~\frac{k_{3}a(b-c)+c(b-a)}{k_{3}(b-c)+(b-a)},~a\right).
    \end{split}
\end{equation}

\end{thm}

\medskip

\noindent {\em Remark:}
If one of $a,b$ or $c$ is $\infty$, (1) can be interpreted suitably by taking limits.
Note also that if $g \neq \rm{I}$, one of $T_1,T_2, T_3$ is at level $n-1$ and the other two are at level $n+1$, so $\TTT$ can be obtained recursively from $\triangle_0=(\infty, -1,0)$ and (1). The modular group case is very special as there is a three-fold symmetry since $k_1=k_2=k_3=1$. In general, one needs to keep track of the labelling of the sides and the parameters $k_1,k_2,k_3$ will  enter into the recursion formula. The recursion for the corresponding punctured torus group and tessellation $\QQQ$ can obtained from this by removing the translates of the geodesic $[\infty, 0]$ by $\Gamma$, or alternatively, by gluing each triangle $(a,b,c)$ with its neighbor $T_3$ to form a quadrilateral.

\medskip

We next explore the question of arithmeticity of the integral Weierstrass groups and jigsaw groups arising from integral jigsaw sets. Recall that the marked triangles $\triangle(1,1/n,n)$, $n \in \NN$ (see section 2.2 for the definition) are called integral jigsaw tiles and sets of integral jigsaw tiles which include the tile $\triangle(1,1,1)$, denoted by $\SSS(1,n_2,\ldots, n_k)$, where $1=n_1 <n_2\cdots<n_k$ are called integral jigsaw sets. We classify all arithmetic integral Weierstrass groups and show that for all but finitely many integral jigsaw sets with two tiles, the corresponding jigsaw groups are always non-arithmetic. In the remaining cases,  the arithmetic jigsaw groups that do arise are finite index subgroups of either star jigsaw groups or diamond jigsaw groups (see definition \ref{def:stardiamond}): 
	
\begin{thm} {\rm (Theorem \ref{thm:nonarith})}
	\begin{itemize}
	    \item [(i)]Suppose that $n \neq 3,5,9,25$. Then any $\SSS(1,n)$ jigsaw group is non-arithmetic. Conversely, if $n\in \{3,5,9,25\}$, then there exists arithmetic $\SSS(1,n)$ jigsaw groups and these arise only as finite index subgroups of  star jigsaw groups  when $n=3$ or $9$, and diamond jigsaw groups when  $n=5$ or $25$.
	    \item [(ii)] The Weierstrass group $\Gamma(1,1/n,n)$ for $n \in \NN$ is arithmetic if and only if $n=1, 2$ or $4$.
	\end{itemize}

\end{thm}

\noindent {\em Remark:} For (i),
\begin{itemize}
\item The cases $n=2,3$ were considered in \cite{LTV}.  When $n=3$, we were able to describe exactly which $\SSS(1,3)$ jigsaw groups were arithmetic. We show here that essentially the same result holds for $n=9$ and analogous results for $n=5,25$, replacing star jigsaws by diamond jigsaws.
\item The case $n=4$ was considered by Garlaz-Garcia recently in \cite{Gar} and she proved Theorems \ref{thm:nonarith} and \ref{thm:infinitepmg} for $\SSS(1,4)$. 
\end{itemize}

If we drop the condition that an integral jigsaw set contains the $\triangle^{(1)}$ tile, the question of arithmeticity  becomes more subtle and interesting. 
By checking the conditions for arithmeticity with a computer code, we show in Proposition \ref{prop:mn} a list of pairs $(m,n)$, $1\le m < n \le 10000$ for which the diamond $\SSS(m,n)$-jigsaw group is arithmetic. In fact, as an application, we can classify all arithmetic surfaces of signature $(0;2,2,2,2;1)$ which admit a reflection symmetry where the fixed point set of the reflection does not contain any of the order two cone points, and all pairs $(m,n)$ giving rise to arithmetic surfaces. Let $\mathcal{M}^{ref}{(0;2,2,2,2;1)}$ denote the moduli space of such hyperbolic surfaces. We have:

\begin{thm}\label{thm:mnintro}{\rm (Theorem \ref{thm:mn})}
 There are exactly eight  arithmetic surfaces in $\mathcal{M}^{ref}{(0;2,2,2,2;1)}$. They correspond to the diamond $\SSS(m,n)$ jigsaw groups for the pairs $$(1,1), (1,5), (2,2), (2,6), (4,4), (4,8), (8,12), (16,20).$$ On the other hand, there are infinitely many pairs $(m,n)\in \NN\times \NN$ for which the diamond $\SSS(m,n)$-jigsaw group is arithmetic. They arise from the action of the mapping class group on the eight basic pairs above.   
\end{thm}


Next, we demonstrate the ubiquity of pseudomodular groups from the integral jigsaw construction (in contrast to the sparcity of arithmetic examples demonstrated by Theorem \ref{thm:nonarith}) by showing that the two jigsaw sets used in \cite{LTV} to construct pseudomodular groups were somewhat arbitrary, and the restriction unnecessary. In fact, starting with {\it any integral jigsaw set}, one can construct infinitely many non-commensurable pseudomodular groups:

\begin{thm}{\rm (Theorem \ref{thm:infinitepmg})}
	For any integral jigsaw set  $\SSS(1,n_2,\ldots, n_s)$, where $1<n_2<\cdots <n_s$,   there exists infinitely many commensurabilty classes of \\
	$\SSS(1,n_2,\ldots, n_s)$  pseudomodular jigsaw groups. 
	
\end{thm}

\medskip
{The Weierstrass groups with rational parameters and the pseudomodular groups share some interesting properties with the modular group, besides having rational cusps. In particular, for any $\alpha \in \RR$, we can associate a unique cutting sequence arising from the group, and also  generalized continued fractions and in some cases, pseudo-Euclidean algorithms. We describe some examples in  section \ref{s:PSAlgorithms}. These algorithms provide convergents $c_i \in \QQ$ to a given $\alpha \in \RR$ compatible with the group. In the case of the diamond $\SSS(1,2)$-pseudomodular group, the continued fraction also has the properties that the partial numerators $a_i \in \{- 1,- 2\}$, the partial denominators $b_i \in \ZZ \setminus \{0\}$, and the complete quotients of rational numbers have decreasing denominators.}

\medskip

Our original motivation for introducing the jigsaw construction was to answer the question of Long and Reid in \cite{LR}. Nonetheless, there are certain aspects of our construction and results which bear some resemblance to the Gromov Piatetski-Shapiro construction of non-arithmetic lattices for Lie groups,  the Grothendieck theory of Dessins d'Enfants and to the Penrose aperiodic tilings of the euclidean plane. It would be interesting to explore possible connections with these areas. It would also be interesting to  see if our constructions and methods can be applied to the study of representations of surface groups to more general Lie groups like $\PSLtwoC$ or $\hbox{PGL}(n,\RR)$, and to explore the dynamics associated to the pseudo-Euclidean algorithms and cutting sequences.

\medskip

The rest of this paper is organized as follows. In section \ref{s:puncturedtorus}, we describe the punctured torus groups, Weierstrass groups, the tessellation and triangulation of $\HH$ associated to these groups and prove Theorem \ref{thm:recursion}. In section \ref{s:jigsaws} we recall briefly the basic definitions and properties of the jigsaw construction from \cite{LTV}. In section \ref{s:nonarith} we prove Theorem  \ref{thm:nonarith} classifying all arithmetic lattices arising from $\SSS(1,n)$ jigsaws and integral Weierstrass groups $\Gamma(1,1/n,n)$, and also Theorem \ref{thm:mn} concerning the arithmetic lattices that arise from $\SSS(m,n)$ diamond jigsaws. In section \ref{s:PMG} we prove Theorem \ref{thm:infinitepmg} showing the existence of infinitely many non-commensurable pseudomodular groups for every integral jigsaw set. In section \ref{s:PSAlgorithms}, we describe  the algorithms for cutting sequences, pseudo-euclidean algorithms and continued fractions associated to some of the groups constructed in the earlier sections.  Finally, in the Appendices, we give some of the details of our proof of arithmeticity for the diamond (for $n=5,25$) and star (for $n=3,9$) jigsaw groups,  list the arithmetic groups arising from  diamond $\SSS(m,n)$, jigsaw groups, where $1 \le m\le n\le 10,000$, and also list some examples of the cutting sequences and continued fractions for various $\alpha \in \RR$ associated to various groups.

\medskip

 \noindent {\it Acknowledgements}. We are grateful to Chris Leininger, Alan Reid, Darren Long and especially Hugo Parlier and Ara Basmajian for their interest in the project, and helpful conversations. 
  Tan was partially supported by the National University of Singapore academic research grant R-146-000-289-114.

\section{Punctured torus groups, marked ideal triangles and Weierstrass groups}\label{s:puncturedtorus}
\subsection{Punctured torus groups and tessellations of $\HH$ by quadrilaterals}
Let $T$ be a hyperbolic torus with one cusp, which can be realized as $\HH/ \Gamma$ where $\Gamma =\langle A,B \rangle$ is a discrete subgroup of $\PSLtwoR$, 
 $A$ and $B$ are hyperbolic elements with intersecting axes and the commutator $[A,B]$ is parabolic. Following \cite{LR}, we may normalize so that a fundamental domain for $T$ is given by the quadrilateral $Q$ with vertex set $\{\infty, -1, 0, u^2\}$,
\begin{equation}\small{
    A=\frac{1}{\sqrt{-1+\tau-u^2}}\left(
\begin{array}{cc}
-1+\tau & u^2 \\
1 & 1 \\
\end{array}
\right),\quad B= \frac{1}{\sqrt{-1+\tau-u^2}}\left(
\begin{array}{cc}
u & u \\
\frac{1}{u} & \frac{(\tau-u^2)}{u} \\
\end{array}
\right),}
\end{equation}
\begin{equation}
  AB^{-1}A^{-1}B= \left(
\begin{array}{cc}
-1 & -2\tau \\
0 & -1 \\
\end{array}
\right),  
\end{equation} 
 and $A$ maps the side  $[-1,0]$ to $[\infty, u^2]$, $B$ maps $[\infty, -1]$ to $[u^2,0]$, and $AB^{-1}A^{-1}B$ fixes $\infty$, (figure 1A).

\begin{figure}
     \centering
     \begin{subfigure}[b]{0.48\textwidth}
         \centering
         \includegraphics[width=\textwidth]{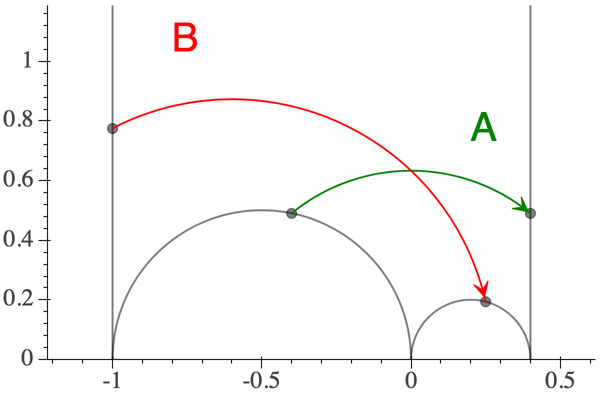}
         \caption{The fundamental domain $Q$}
         \label{fig:y equals x}
     \end{subfigure}
     \begin{subfigure}[b]{0.48\textwidth}
         \centering
         \includegraphics[width=\textwidth]{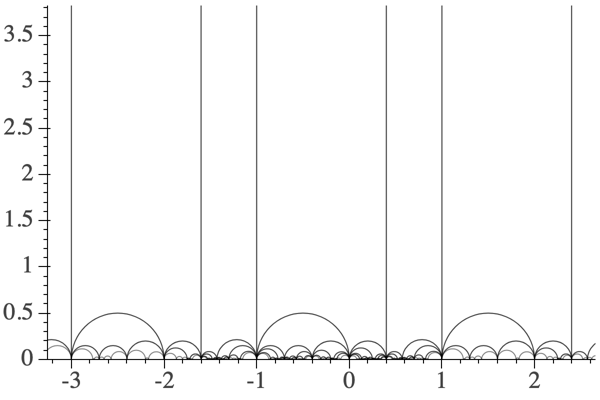}
         \caption{The tessellation $\QQQ_{(2/5,4)}$}
         \label{fig:three sin x}
     \end{subfigure}
        \caption{The fundamental domain $Q$ and tessellation $\QQQ_{(2/5,4)}$ corresponding to the pseudomodular group $\Delta(2/5, 4)$.}
        \label{fig:Qandtessellation}
\end{figure}

The group $\Gamma$ with parameters $u^2$ and $\tau$, $\tau>1+u^2$, is denoted by $\Delta(u^2, 2 \tau)$ in \cite{LR} and the translates of $Q$ by $\Delta(u^2, 2\tau)$ gives a tessellation $\QQQ_{(u^2,2\tau)}$ of $\HH$ by ideal quadrilaterals (figure 1B). An important special case is the group $\Delta(1,6)$ which is a finite index subgroup of the modular group and in this case $\QQQ_{(1,6)}$ is a subset of the Farey triangulation obtained by removing the translates of $[0,\infty]$ by $\Delta(1,6)$. In particular, in this case, one has a recursive way of describing $\QQQ_{(1,6)}$ by using the standard ``Farey addition'' to construct the Farey tessellation. Of interest too are the groups where $u^2\in \QQ$ and $\tau \in \NN$, (where $\Delta(u^2,2\tau) <\PGLtwoQ$) and especially those parameters which give rise to pseudomodular groups (for example, $\Delta(2/5, 4)$ used in figure 1) since in the latter case, the set of vertices of $\QQQ_{(u^2,2\tau)}$ is precisely $\QQ \cup \{\infty\}$. Long and Reid asked  \cite[Question 5, section 6, see also section 3.3]{LR}  if there was a recursive formula for $\QQQ_{(u^2, 2\tau)}$. We will answer this question by considering the supergroup of $\Delta(u^2, 2\tau)$  obtained by including the hyperelliptic involution which gives rise to what we call Weierstrass groups, which we describe next.

\subsection{Weierstrass groups and triangulations of $\HH$} Every cusped hyperbolic torus $T=\HH/\Gamma$ admits a hyperelliptic involution, the quotient surface $P$ has signature $(0;2,2,2;1)$. The corresponding supergroup of $\Gamma$ is a Coxeter group generated by the order two involutions ($\pi$-rotations) about the three Weierstrass points on the torus. Using the normalization for $T$ in the previous subsection, the fundamental domain for $P$ can be taken to be the ideal triangle $\triangle_0=(\infty, -1,0)$  and the Weierstrass points project to the marked points 
\begin{equation}\label{eqn:fixedpoints}
x_1=-1+\frac{i}{\sqrt{k_1}}, \quad x_2=\frac{-1+\sqrt{k_2}i}{1+k_2}, \quad x_3=\sqrt{k_3}i
\end{equation}  on the sides $[\infty, -1]$, $[-1,0]$ and $[0,\infty]$ of $\triangle_0$ respectively, where
\begin{equation}
u^2=k_3, \qquad \tau=1/k_1+1+k_3, \qquad k_1k_2 k_3=1,
\end{equation}
or equivalently,
\begin{equation}
k_1=\frac{1}{\tau-1-u^2}, \qquad k_2=\frac{\tau-1-u^2}{u^2}, \qquad k_3=u^2.
\end{equation}
Denote by $\triangle(k_1,k_2,k_3)$ the ideal triangle $\triangle_0$ with the marked points $x_1,x_2,x_3$ given by (\ref{eqn:fixedpoints}), up to $\hbox{Isom}^+(\HH)$, see figure \ref{fig:triangle}.
The $\pi$-rotations (involutions) $\iota_1$, $\iota_2$ and $\iota_3$ about the Weierstrass points $x_1,x_2$ are $x_3$ respectively are 
\begin{equation}\label{eqn:involutions}
\Small
\iota_1 = \frac{1}{\sqrt{k_1}}\left(
\begin{array}{cc}
k_1 & 1+k_1 \\
-k_1 & -k_1 \\
\end{array}
\right), \enskip
\iota_2 =
\frac{1}{\sqrt{k_2}} \left(
\begin{array}{cc}
1 & 1 \\
-(k_2+1) & -1 \\
\end{array}
\right), \enskip
\iota_3 =
\frac{1}{\sqrt{k_3}} \left(
\begin{array}{cc}
0 & k_3 \\
-1 & 0 \\
\end{array}
\right).\end{equation}

\medskip

\medskip

We call the Coxeter group $\Gamma(k_1,k_2,k_3):=\langle \iota_1, \iota_2, \iota_3 ~|~ \iota_1^2=\iota_2^2=\iota_3^2=\hbox{I} ~\rangle$ the Weierstrass group with parameters, $k_1,k_2,k_3 \in \RR_+$, where $k_1k_2k_3=1$. The translates of $\triangle_0$ by $\Gamma(k_1,k_2,k_3)$ gives a triangulation $\TTT_{(k_1,k_2,k_3)}$ of $\HH$ (Figure \ref{fig:triangulation}). The relation to the tessellation $\QQQ_{(u^2, 2\tau)}$ is that $\QQQ_{(u^2, 2\tau)}$ is obtained from  $\TTT_{(k_1,k_2,k_3)}$ by removing the translates of the geodesic $[0, \infty]$ by $\Gamma(k_1,k_2,k_3)$. Hence a recursive formula for the triangulation essentially gives a recursive formula for the tessellation. For example, the group $\Delta(2/5,4)$ used in Figure \ref{fig:Qandtessellation} is a subgroup of the group $\Gamma(5/3,3/2,2/5)$ used in Figure \ref{fig:Tandtriangulation}.

\begin{figure}
     \centering
     \begin{subfigure}[b]{0.48\textwidth}
         \centering
         \includegraphics[width=\textwidth]{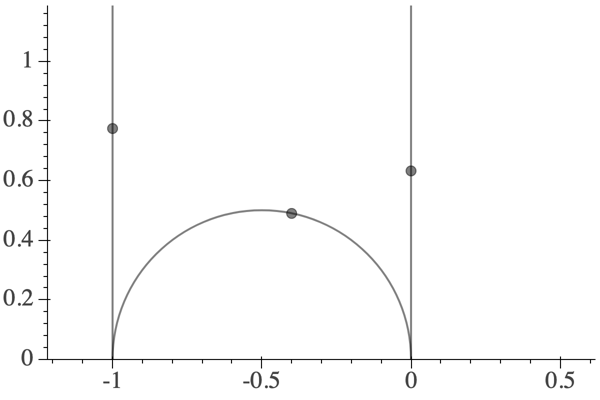}
         \caption{ $\triangle_0$ with marked points $x_1,x_2,x_3$.}
         \label{fig:triangle}
     \end{subfigure}
     \begin{subfigure}[b]{0.48\textwidth}
         \centering
         \includegraphics[width=\textwidth]{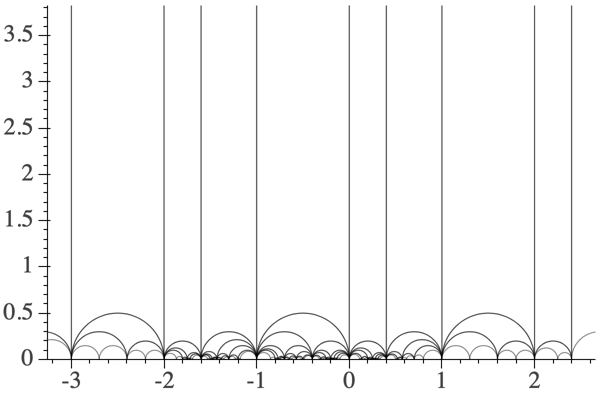}
         \caption{The triangulation $\TTT_{(5/3,3/2,2/5)}$}
         \label{fig:triangulation}
     \end{subfigure}
        \caption{The fundamental domain $\triangle_0$ and triangulation $\TTT_{(5/3,3/2,2/5)}$ corresponding to the Weierstrass group $\Gamma(\frac{5}{3},\frac{3}{2},\frac{2}{5})$.}
        \label{fig:Tandtriangulation}
\end{figure}

\medskip

If $g \in \Gamma{(k_1,k_2,k_3)}$ has reduced word length $n$ in the generators $\iota_1,\iota_2, \iota_3$, we say that the triangle 
$$(a,b,c):=(g(\infty), g(-1), g(0)),$$ 
which is the translate of $\triangle_0$ by $g$ is a {\em level $n$ triangle}. Apart from $\triangle_0$, every level $n$ triangle in $\TTT_{(k_1,k_2,k_3)}$ is adjacent to two triangles of level $n+1$ and one of level $n-1$ and there is a unique path from $\triangle_0$ to any such $(a,b,c)$ along adjacent triangles in $\TTT_{(k_1,k_2,k_3)}$ with the levels increasing by $1$ along each step. These notions can be seen more clearly by looking at the dual graph $G$ to the triangulation $\TTT$ which is an infinite trivalent tree where the vertices of the tree correspond to the triangles of $\TTT$, and edges have length 1. The level of a particular triangle is then just the distance of the corresponding vertex $v\in V(G)$ from the vertex $v_0\in V(G)$  corresponding to $\triangle_0$. We have:

\begin{thm}\label{thm:recursion}
Let $(a,b,c)$ be a level $n$ triangle in $\TTT_{(k_1,k_2,k_3)}$ and let $T_1$, $T_2$ and $T_3$ be the triangles adjacent to $(a,b,c)$, where $T_i$ shares the side marked by $k_i$ with $(a,b,c)$. Then
\begin{equation}\label{eqn:recursion}
\begin{split}
    T_{1} & = \left(b,~a,~~\frac{k_{1}b(c-a)+a(c-b)}{k_{1}(c-a)+(c-b)}\right)
     \\
     T_{2} &=\left(\frac{k_{2}c(a-b)+b(a-c)}{k_{2}(a-b)+(a-c)},~~c,~b\right)
    \\
    T_{3} &= \left(c,~~\frac{k_{3}a(b-c)+c(b-a)}{k_{3}(b-c)+(b-a)},~a\right).
    \end{split}
\end{equation}

\end{thm}

If the parameters are rational, and we write $k_1=p/q$, then we have $$T_1=\left(b,a,\frac{pb(c-a)+qa(c-b)}{p(c-a)+q(c-b)}\right)$$
with similar expressions for $T_2$ and $T_3$. When $k_1=k_2=k_3=1$, this reduces to the standard Farey addition.

Note that if one of $a,b$ or $c$ is $\infty$, the above is defined as for cross ratios by taking the relevant limits, for example, if $(a,b,c)=(\infty,b,c)$, then

$$T_{1}=\left(b,\, \infty,\, \frac{k_{1}b+b-c}{k_{1}}\right), \quad T_{2}=\left(\frac{b+k_{2}c}{1+k_{2}},c,b\right), \quad T_{3}=\left(c,c-{k_{3}(b-c)},\infty\right).$$

\begin{proof}
We first assume $a,b,c \neq \infty$ and let $y_1$ be the marked point on the side $[a,b]$ of the triangle $(a,b,c)$,  $h_1$ the $\pi$-rotation about $y_1$ and let $g \in \Gamma(k_1,k_2,k_3)$ such that $g(\infty, -1,0)$=$(a,b,c)$. Then $h_1=g\iota_1g^{-1}$ and $T_1=h_1(a,b,c)=g\iota_1(\infty, -1,0)$. By a direct calculation, $$g=\frac{1}{\sqrt{(a-b)(b-c)(c-a)}}\left(\begin{array}{cc}
a(c-b) & c(a-b)\\
c-b & a-b
\end{array}\right)$$ 
here $(a-b)(b-c)(c-a)$ must be positive as $g\in \PSLtwoR$. From this, we get $$h_1(a)=g\iota_1(\infty)=b, \qquad \quad h_1(b)=g\iota_1(-1)=a, ~~$$
$$h_1(c)=g\iota_1(0)=\frac{k_{1}b(c-a)+a(c-b)}{k_{1}(c-a)+(c-b)}.$$
The formulas for $T_2$ and $T_3$, and the cases where one of $a,b,c=\infty$ can be obtained similarly, details are left to the reader.

\end{proof}

\subsection{Examples} The case $\Gamma(1,1,1)$ gives rise to the Farey sequence $\{F_n\}$ which can be defined as follows, listing the vertices of all the triangles in $\TTT_{(1,1,1)}$ up to level $n-1$:
\begin{eqnarray*}
    F_1:=\{-\infty, -1,0, \infty\}, \quad F_2:\{-\infty, -2,-1, -1/2,0, 1, \infty\}, \\
    F_3=\{-\infty, -3,-2,-3/2,-1,-2/3, -1/2, -1/3,0, 1/2,1,2, \infty\},\ldots
\end{eqnarray*}
In general, $F_{n+1}$ can be obtained from $F_n$ by inserting new terms between successive terms of $F_n$ using the ``Farey addition'' of numerators and denominators. For the $n$th Farey sequence of $\Gamma(k_1,k_2,k_3)$ which we denote by $F_n(k_1,k_2,k_3)$ we use the recursion defined by (\ref{eqn:recursion}).

\medskip

\noindent { For $\Gamma(1,1/5,5)$}, we have:
$$F_1(1,1/5,5)=\{-\infty, -1,0, \infty\},\quad 
F_2(1,1/5,5)=\{-\infty, -2, -1, -5/6, 0, 5, \infty\},\ldots$$
$F_5(1,1/5,5)=\{-\infty,-9,-8,\frac{-47}{6},-7,\frac{-44}{7},\frac{-9}{2},\frac{-19}{7},-2,\frac{-19}{12},\frac{-9}{7},\frac{-44}{37},\frac{-7}{6},\frac{-47}{41},\frac{-8}{7},\\
\frac{-9}{8},-1,\frac{-12}{13},\frac{-7}{8},\frac{-37}{43},\frac{-6}{7},\frac{-41}{48},\frac{-35}{41},\frac{-40}{47},\frac{-5}{6},\frac{-35}{43},\frac{-30}{37},\frac{-35}{44},\frac{-5}{7},\frac{-10}{19},\frac{-5}{12},\frac{-5}{13},0,\frac{5}{8},\frac{5}{7},\\ \frac{10}{9},\frac{5}{2},\frac{35}{9},\frac{30}{7},\frac{35}{8},5,\frac{40}{7},\frac{35}{6},\frac{41}{7},6,\frac{37}{6},7,12, \infty\},\ldots $

As an example of getting from $F_4(1,1/5,5)$ to $F_5(1,1/5,5)$, consider the level $3$ triangle $(-\frac{6}{7}, -1, -\frac{7}{8})$. By (\ref{eqn:recursion}) the adjacent triangles are:
\begin{equation*}
\begin{split}
    T_{1} &= \left(~-1, ~~-\frac{6}{7},~~-\frac{5}{6}\right) \qquad \hbox{(Level $2$)}.
    \\
    T_{2} & = \left(~-\frac{12}{13},~~-\frac{7}{8},~~-1\right) \qquad \hbox{(Level $4$)}
     \\
    T_{3} &=\left(~-\frac{7}{8},~~-\frac{37}{43},~~-\frac{6}{7}\right) \qquad \hbox{(Level $4$)}
    \end{split}
\end{equation*}

\medskip
\noindent {For $\Gamma(\frac{5}{3},\frac{3}{2}, \frac{2}{5})$}, we have:
$$F_1(\frac{5}{3},\frac{3}{2}, \frac{2}{5})=\{-\infty, -1,0, \infty\}, \quad
F_2(\frac{5}{3},\frac{3}{2}, \frac{2}{5})=\{-\infty, \frac{-8}{5},-1,\frac{-2}{5},0,\frac{2}{5}, \infty\},\ldots$$ 
$F_5(\frac{5}{3},\frac{3}{2}, \frac{2}{5})=\{-\infty, \frac{-18}{5},-3,\frac{-12}{5},-2,\frac{-9}{5},\frac{-7}{4},\frac{-43}{25},\frac{-8}{5},\frac{-3}{2},\frac{-37}{25},\frac{-47}{32},\frac{-10}{7},\frac{-76}{55},\frac{-13}{10},\\
\frac{-6}{5}, -1,\frac{-4}{5},\frac{-10}{13},\frac{-64}{85},\frac{-7}{10},\frac{-11}{17},\frac{-4}{7},\frac{-1}{2},\frac{-2}{5},\frac{-4}{13},\frac{-1}{4},\frac{-2}{9},\frac{-1}{5},\frac{-8}{45},\frac{-1}{6},\frac{-2}{15},0,\frac{2}{15},\frac{1}{5}, \frac{8}{35},\frac{1}{4},\\
\frac{10}{37},\frac{7}{25},\frac{4}{13},\frac{2}{5},\frac{13}{25},\frac{4}{7},\frac{7}{10},1,\frac{8}{5},2,\frac{12}{5}, \infty\}$,...

\medskip

As in the previous example, consider the level $3$ triangle 
$(-\frac{7}{10}, -1, -\frac{10}{13})$.
By (\ref{eqn:recursion}) the adjacent triangles are:
\begin{equation*}
\begin{split}
    T_{1} &= \left(~-1,~~-\frac{7}{10},~~-\frac{2}{5}\right) \qquad \hbox{(Level $2$)}.
    \\
    T_{2} & = \left(~-\frac{4}{5},~~-\frac{10}{13},~~-1\right) \qquad \hbox{(Level $4$)}
     \\
    T_{3} &=\left(~-\frac{10}{13},~~-\frac{64}{85},~~-\frac{7}{10}\right) \qquad \hbox{(Level $4$)}
    \end{split}
\end{equation*}

Note that the vertex emerging at a higher level may have a smaller denominator than the corresponding vertex at the previous level, as for $T_2$ in this example where $-\frac{4}{5}$ emerges as an image of $-\frac{7}{10}$. Also, this group is pseudomodular, so all rationals are cusps, although a rational with relatively small denominator may be at a fairly high level. For example, $-\frac{1}{3}$ first appears as a vertex of the level 7 triangle $(-\frac{13}{40}, -\frac{1}{3}, -\frac{2}{5})$.



\section{Jigsaws and integral jigsaw groups}\label{s:jigsaws}
We recall briefly the definitions and objects associated to the jigsaw construction and some of the basic properties needed in later sections.  The reader is referred to sections 2 and 4 of \cite{LTV} for more details.

\subsection{Basic definitions}\label{ss:defn}
\begin{defn}(Integral jigsaw tiles,  jigsaw sets, jigsaws and jigsaw groups)
\begin{enumerate}
        \item The marked triangles $\triangle^{(n)}:=\triangle(1,1/n,n)$, $n\in \NN$ are called {\it integral jigsaw tiles}, and sets of integral tiles of the form   
$$\SSS(1, n_2, \ldots, n_s)=\{\triangle^{(1)}, \triangle^{(n_2)}, \ldots, \triangle^{(n_s)} \}, \quad 1=n_1<n_2< \cdots <n_s$$ 
are called {\it integral jigsaw sets}. 
\item An ideal polygon $J$ which is assembled from the tiles of $\SSS(1, n_2, \ldots, n_s)$ by gluing tiles  along sides with the same labels,  such that the marked points match up, is called an (integral) $\SSS(1,n_2, \ldots, n_s)$-jigsaw, provided that all tiles are used at least once. Its signature is defined to be $$sgn(J):=(m_1, \ldots, m_s)\in \NN^s,$$ where $m_k $ tiles of type $n_k$ are used to assemble $J$, $k=1, \ldots, s$. The {\it size} of $J$, denoted by $|J|$ is the total  number of tiles in $J$. By definition,  $\sum_{k=1}^s m_k=N:=|J|$ and $J$ is an ideal $(N+2)$-gon.

\item The jigsaw $J$ is in {\it normalized position} in $\HH$ if the triangle $\triangle=(\infty, -1,0)$ is a $\triangle^{(1)}$ piece of $J$. We always assume $J$ is in normalized position, unless otherwise stated. 

\item Denote the vertices of $J$ in counterclockwise direction by $v_0=-\infty< v_1< \cdots< v_{N+1}<v_{N+2}=v_0=\infty$ and let  $x_j \in s_j:= [v_j,v_{j+1}]$ be the marked points on the sides $s_j$ of $J$, where $j=0,1, \ldots, N+1$. The Coxeter group $\Gamma_J:=\langle \iota_0, \ldots, \iota_{N+1} \rangle$ where $\iota_j$ is the $\pi$-rotation about the marked point $x_j$ is the jigsaw group associated to the jigsaw $J$. We call $\Gamma_J$ an (integral) $\SSS(1,n_2, \ldots, n_k)$-jigsaw group. Note that $J$ is a fundamental domain for $\Gamma_J$. The surface $\HH/\Gamma_J$ has signature $(0; \underbrace{2, \ldots, 2}_{N+2} ;1)$.

\item  The product of the generators is parabolic, more precisely,	\[\imath_{N+1}\imath_N\cdots\imath_{0}=\pm \left(
			\begin{array}{cc}
			1 & L \\
			0 & 1 \\
			\end{array}
			\right) \] where \[L=\sum_{i=1}^{s}m_i(2+n_i)\in \NN.\] 
	 An interval of the form $(x,x+L]\subset \RR$ is called a {\it fundamental interval} for the group $\Gamma_J$.

    \end{enumerate}
\end{defn}

  \begin{defn} (Tessellations, triangulations, labels, types, vertex weights and vertex $J$-widths).
  \begin{enumerate}
      \item  Associated to $J$ and $\Gamma_J$ is a tessellation $\QQQ_J$ of $\HH$ by the translates under $\Gamma_J$ of $J$, and also a triangulation $\TTT_J$ of $\HH$ by the translates under $\Gamma_J$ of the  triangulation of $J$ by its tiles.
      \item The triangles of $\TTT_J$ are labelled by elements of the set $\{1,n_2, \ldots, n_s\}$ coming from the labelling of the tiles. The  sides of the tiles are labelled by elements of the set $\{1, 1/n_2,n_2, \ldots, 1/n_s, n_s\}$. This induces a labelling of the edges of $\TTT_J$ and $\QQQ_J$. An edge  of $\QQQ_J$ or $\TTT_J$ is said to be of {\it type} $n$ if its label is $n$ or $1/n$.
      \item The weight of a vertex $v$ of $J$, denoted by $\hbox{wt}(v)$ is the number of triangles at that vertex.
      \item The $J$ width of a vertex of a $\triangle ^{(n)}$ tile is $n$ if the vertex is between two type $n$ sides and $1$ otherwise. The $J$ width of a vertex $v_k$ of $J$, denoted $\hbox{JW}(v_k)$, is the sum of the $J$-widths of the vertices of the tiles at $v_k$.
      In particular, the length of the fundamental interval $L=\sum_{k=0}^{N+1}\hbox{JW}(v_k).$
  \end{enumerate}
     \end{defn}

  \begin{figure}
     \centering
     \begin{subfigure}[b]{0.48\textwidth}
         \centering
         \includegraphics[width=\textwidth]{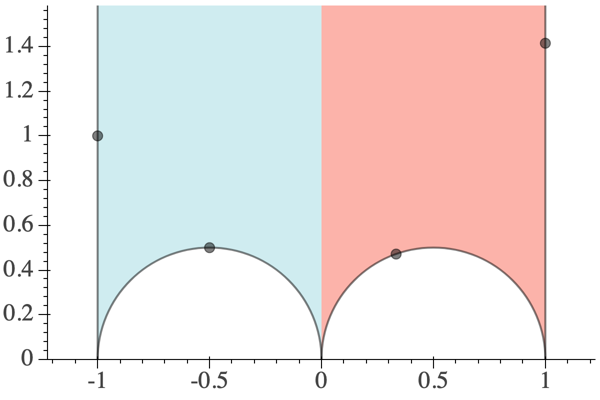}
         \caption{Left tile is $\triangle^{(1)}$, right tile is $\triangle^{(2)}$.}
         \label{fig:y equals x}
     \end{subfigure}
     \begin{subfigure}[b]{0.48\textwidth}
         \centering
         \includegraphics[width=\textwidth]{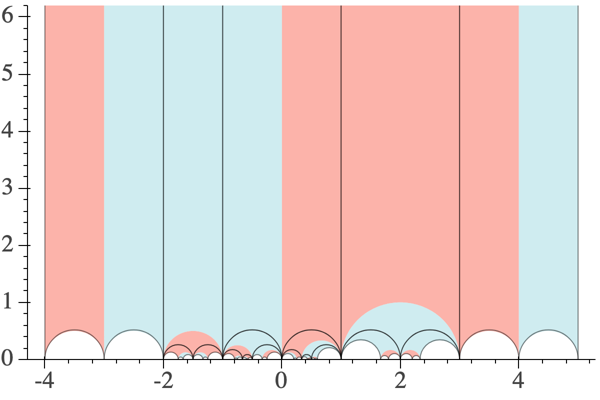}
         \caption{The tessellation $\QQQ_J$ (edges are black)}
         \label{fig:three sin x}
     \end{subfigure}
        \caption{The (diamond) $\SSS(1,2)$ jigsaw with two tiles, and associated tessellation}
        \label{fig:three graphs}
\end{figure}

  \subsection{Basic properties}\label{ss:basicproperties}
  The set of edges of $\QQQ_J$ is naturally a subset of the set of edges of $\TTT_J$ while they both share the same set of vertices, which is precisely the cusp set of $\Gamma_J$. For $j \in \ZZ$, let $e_j=[\infty, m_j]$ be successive (leftwards) vertical edges of $\QQQ_J$ where we fix $e_0$ to be the side $[\infty, v_1]$ of $J$, so $e_{-1}=[\infty, v_{N+1}]$, where $|J|=N$.
  The key facts (proven in \cite{LTV}) we will use about the triangulation $\TTT_J$,  tessellation $\QQQ_J$ and the vertical edges $\{e_j\}$ (see Figure 3 for an illustration), can be summarized as:

  \begin{lem}\label{lem:QTV} \cite[Proposition 4.3, 4.5, Lemma 7.1]{LTV} For an integral $\SSS(1,n_2, \ldots, n_s)$-jigsaw $J$ in normalized position, with $|J|=N$, the triangulation $\TTT_J$, tessellation $\QQQ_J$ and vertical edges $e_k=[\infty, m_k]$ of $\QQQ_J$ satisfy:
       
 \begin{enumerate}
      \item If $(\infty, x,y)$ is a triangle of $\TTT_J$, then $x$ and $y$ are integers that differ by either $1$ or $n_j$, the latter occurring only if $(\infty, x,y)$ is a $\triangle^{(n_j)}$-tile with label $1$ opposite the $\infty$ vertex. In this case, the triangles adjacent to $(\infty, x,x+n_j)$ are $(\infty, x-1,x)$ and $(\infty, x+n_j,x+n_j+1)$.
      \item The real end point $m_k$ of $e_k$ is an integer for all $k\in \ZZ$.
      \item If $e_k$ is of type $n_j$, then the marked point on $e_k$ is at $m_k+\sqrt{n_j}i$ and the $\pi$-rotation $h_k$ about the marked point is given by  $$h_{k}=\frac{1}{\sqrt{n_j}}\left(\begin{array}{cc}
m_{k} & -(m_{k}^{2}+n_j)\\
1 & -m_{k}
\end{array}\right).$$

      \item The difference $|m_k-m_{k+1}|=JW(v_{k+1})$ where the indices are taken $\mod N+2$.
      \item $m_k-m_l\equiv 0 \mod L$ if $k -l \equiv 0\mod N+2$, where  $L=m_{0}-m_{N+2}$  is the length of the fundamental interval for $\Gamma_J$. Here $e_k$ and $e_l$ are lifts of the same side of $J$.
      \item For each $n\in\{n_{2},\ldots,n_{s}\}$, at least two of the sides
$s_{k}$, $s_{l}$ of $J$ are type $n$. 
      
  \end{enumerate}
  
   \end{lem}

 The above suffices for our purposes, for more properties, see Propositions 4.3, 4.5 and 4.6 of \cite{LTV}. We will see in the next section that arithmeticity of $\Gamma_J$ forces certain congruences on the endpoints $m_k$ of the vertical edges of $\QQQ_J$ which are often not satisfied.

\section{Arithmeticity of integral Jigsaw and Weierstrass groups}\label{s:nonarith} We classify all arithmetic $\SSS(1,n)$-jigsaw groups and integral Weierstrass groups, and also all arithmetic surfaces of signature $(0;2,2,2,2;1)$ admitting a reflection symmetry.

\subsection{Statement of results for arithmeticity and strategy of proof}\label{ss:resultsarith}
\begin{defn}\label{def:stardiamond}(Diamond and star jigsaws).
    Fix the jigsaw set $\SSS(1,n)$. The {\em diamond jigsaw} is the jigsaw consisting of one $\triangle^{(1)}$ tile attached to a $\triangle^{(n)}$ tile, and the {\em star jigsaw} is the jigsaw consisting of one $\triangle^{(1)}$ tile in the middle, and three $\triangle^{(n)}$ tiles attached to its three sides, see figure 4. The corresponding jigsaw groups $\Gamma_J$ are called {\em diamond} and {\em star} jigsaw groups respectively. The diamond jigsaw for the set $\SSS(m,n)$ is defined similarly.
\end{defn} 


\begin{figure}
     \centering
     \begin{subfigure}[b]{0.48\textwidth}
         \centering
         \includegraphics[width=\textwidth]{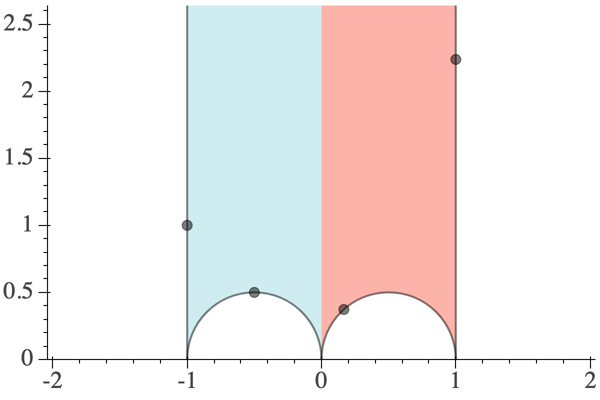}
         \caption{Diamond jigsaw, left tile is $\triangle^{(1)}$ tile}
         \label{fig:y equals x}
     \end{subfigure}
     \begin{subfigure}[b]{0.48\textwidth}
         \centering
         \includegraphics[width=\textwidth]{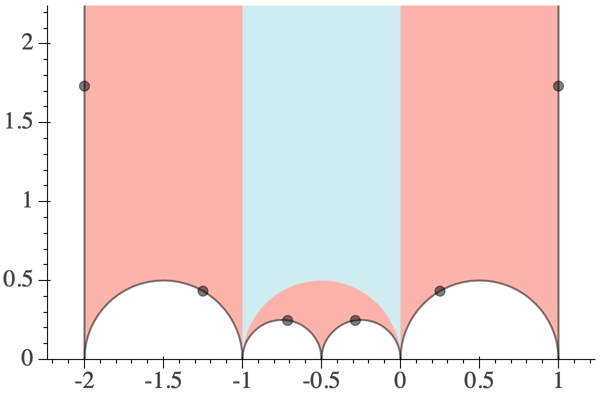}
         \caption{Star  jigsaw, center tile is $\triangle^{(1)}$ tile.}
         \label{fig:three sin x}
     \end{subfigure}
        \caption{The diamond $\SSS(1,5)$ jigsaw and the star $\SSS(1,3)$ jigsaw. }
        \label{fig:three graphs}
\end{figure}

\begin{thm}\label{thm:nonarith}~~~
	\begin{itemize}
	    \item [(i)] Suppose that $n \notin \{3,5,9,25\}$. Then any $\SSS(1,n)$ jigsaw group is non-arithmetic. Conversely, if $n\in \{3,5,9,25\}$, then there exists arithmetic $\SSS(1,n)$ jigsaw groups and these arise only as finite index subgroups of  star jigsaw groups  when $n=3$ or $9$, and diamond jigsaw groups when  $n=5$ or $25$.
	    \item [(ii)] The Weierstrass group $\Gamma(1,1/n,n)$ for $n \in \NN$ is arithmetic if and only if $n=1, 2$ or $4$.
	\end{itemize}
\end{thm}

We divide the statement and proof of (i) into several different propositions, adopting the notation of the previous section. It will be convenient to work with  $H_J <\Gamma_J$, the index two subgroup consisting of words of even length. If $\iota_0, \ldots, \iota_{N+1}$ are the generators of $\Gamma_J$ where $\iota_k$ are $\pi$-rotations about the marked points $x_k$ on the sides $s_k$ of $J$, listed in cyclic order, then $H_J$ is the free group
$$H_J=\langle \gamma_1, \dots, \gamma_{N+1}\rangle \quad \hbox{where} \quad \gamma_k=\iota_0\iota_k,~ k=1, \ldots, N+1. $$ 
To check for arithmeticity of $H_J$, by \cite{Tak}, it suffices to check that $\hbox{tr}\, \gamma^2 \in \ZZ$ (equivalently, $(\tr \gamma)^2 \in \ZZ$), for all $\gamma \in H_J$. In fact, by \cite{HLM} and \cite{Hor}, it suffices to check that this holds for a set of $2^{N+1}-1$ words in $\{\gamma_k\}$, and that $\tr \gamma \in \ZZ$ for the same set of words in $\{\gamma_k^2\}$, $k=1, \ldots, N+1$. The strategy will be  to show that in almost all cases, arithmeticity would imply certain congruences which cannot be satisfied, and to  construct arithmetic examples for the remaining cases. 

\subsection{Admissible jigsaws and triangulated polygons}\label{ss:admissibility}
The following definition will be useful:

\begin{defn} (Admissible jigsaws)
    Let $J$ be a $\SSS(1, n_2, \ldots, n_s)$-jigsaw in normalized position and let $e_k=[\infty, m_k]$, $k \in \ZZ$ be the vertical sides of $\QQQ_J$, $h_k$ the $\pi$-rotation about the marked point on $e_k$, as used in Lemma \ref{lem:QTV}(3). $J$ is {\em admissible} if $(\tr (h_kh_l))^2 \in \ZZ$ for all $k,l \in \ZZ$.
\end{defn}

By the criteria for arithmeticity, if $H_J$ is arithmetic, then $J$ is admissible since $h_kh_l \in H_J$ for all $k,l \in \ZZ$. On the other hand, the condition forces certain congruences which in general will not be satisfied.

Explicitly,  let $e_k=[\infty, m_k]$ and $e_l=[\infty, m_l]$ be two vertical sides of $\QQQ_J$ of type $u$ and $v$ respectively. Using the formula of the $\pi$-rotations $h_{k}$ and $h_l$ in Lemma \ref{lem:QTV}(3), we get 
\begin{equation}\label{eqn:tracesquare}
    (\tr (h_{k}h_{l}))^2 = \dfrac{(-(m_{k}-m_{l})^2-u-v)^2}{uv}=\dfrac{((m_{k}-m_{l})^2+u+v)^2}{uv}.
\end{equation}

In particular, applying (\ref{eqn:tracesquare}) to $\SSS(1,n)$-jigsaws, we have:

\begin{lem}\label{lem:congruences} If $J$ is a $\SSS(1,n)$ jigsaw and $e_k=[\infty, m_k]$, $k \in \ZZ$, the vertical sides of $\QQQ_J$, then:
\begin{enumerate}
    \item { If both $e_{k}$ and $e_{l}$ are type $1$, then} $$\left[tr\left(h_{k}h_{l}\right)\right]^{2}=\left(\left(m_{k}-m_{l}\right)^{2}+2\right)^{2}$$
which is always an integer.

\item { If both $e_{k}$ and $e_{l}$ are type $n$, then} $$\left[tr\left(h_{k}h_{l}\right)\right]^{2}=\frac{1}{n^{2}}\left(\left(m_{k}-m_{l}\right)^{2}+2n\right)^{2}\in\mathbb{Z}$$
if and only if $n|(m_{k}-m_{l})^{2}$.

\item 
{ If  $e_{k}$ and $e_{l}$ are different types, then} 
$$\left[tr\left(h_{k}h_{l}\right)\right]^{2}=\frac{1}{n}\left(\left(m_{k}-m_{l}\right)^{2}+1+n\right)^{2}\in\mathbb{Z}$$
if and only if  $n|((m_{k}-m_{l})^{2}+1)^{2}$.

\end{enumerate}

\end{lem}

We will also need some basic facts about triangulated polygons $P$, which will apply to the jigsaws $J$. 

\begin{defn}
    We call a triangle $\triangle$ a {\em ear} of $P$ if two of the sides of $\triangle$ are sides of $P$ and the vertex between these two sides the {\it tip} of the ear. 
\end{defn}

\begin{lem}\label{lem:polygon}\cite[Lemma 7.4]{LTV} Suppose $P$ is a triangulated polygon with at least 4 vertices.
\begin{itemize}
    \item[(a)] $P$ has at least two ears, hence  at least two vertices $v_i$, $v_j$ with $\hbox{wt}(v_i)=\hbox{wt}(v_j)=1.$
    \item[(b)]   There are at least two non-adjacent vertices $v_j$, $v_k$ of $P$ with $\hbox{wt}(v_j)$, $\hbox{wt}(v_k) \in \{2,3\}$. 
\end{itemize}
    
    \end{lem}

The following is an easy consequence of Lemmas \ref{lem:congruences} and \ref{lem:polygon}:
\begin{cor}\label{cor:wt}
If $J$ is a $\SSS(1,n)$ jigsaw, then 
\begin{itemize}
    \item it has two non-adjacent vertices $v_i,v_j$ such that $\hbox{JW}(v_i), \hbox{JW}(v_j) \in \{2,3,n+1,n+2\}$.
    \item the tip of a ear of $J$ cannot be between a type 1 side and a type $n$ side, if $J$ is admissible and $n \notin \{2,4\}$.
\end{itemize}

\end{cor}
\subsection{Subcases of Theorem \ref{thm:nonarith}}\label{ss:subcases}
We  first consider $\SSS(1,n)$ with $n$  even.  We have:

\begin{prop}\label{prop:even}
    If $n$ is even, then all $\SSS(1,n)$ jigsaw groups are non-arithmetic.
\end{prop}

\begin{proof}
Essentially the same proof as the proof of proposition 7.2  of \cite{LTV} for $\SSS(1,2)$  jigsaws works. For the convenience of the reader, and as  elements of the same arguments are used in later cases, we repeat it here. 

Suppose not. Then we have an $\SSS(1,n)$ jigsaw $J$ which is admissible.
For any two vertical sides $e_{k}=[\infty, m_{k}]$  and $e_{l}=[\infty, m_{l}]$ of $\QQQ_J$, by Lemma \ref{lem:congruences}, if both $e_{k}$ and $e_{l}$ are type $n$, then $n|(m_{k}-m_{l})^2$ which implies $m_{k}-m_{l}$ is even.
If $e_{k}$ and $e_{l}$ are different types, then $n|((m_{k}-m_{l})^2+1)^2$ which implies $m_{k}-m_{l}$ is odd. 
Furthermore, by Lemma \ref{lem:QTV} (6), at least two sides of $J$ are type $n$. Hence $m_{k}-m_{l}$ is even if both $e_{k}$ and $e_{l}$ are type 1. 

Now pick an ear of $J$, it must be of type $n$,  otherwise we have two type 1 sides with different parity. Let $J_2$ be the jigsaw obtained by removing the ear. It is easy to verify that $J_2$ will still satisfy the same parity conditions for its sides, with respect to the type. In particular, it cannot have only type 1 tiles (otherwise there will be two type 1 sides with different parity). So $J_2$ is admissible and we can proceed inductively. However, after a finite number of steps we must end with $J_n$ consisting only of type 1 tiles, (since we assume at least one type 1 tile in the definition) which is not admissible, giving a contradiction.
\end{proof}

\noindent We next show:

\begin{prop}\label{prop:othercases}
    For $n\neq 3,5,9,25$ and odd, all $\SSS(1,n)$ jigsaw groups are non-arithmetic.
\end{prop}

\begin{proof} We will show that the jigsaws are not admissible by contradiction, eliminating all cases systematically. Suppose $J$ is admissible.
By Lemma \ref{lem:congruences}, if both $e_{k}$ and $e_{l}$ are type $n$, then $n|\left(m_{k}-m_{l}\right)^{2}$
which implies $m_{k}-m_{l}$ is divisible by all prime factors of $n$. If $e_{k}$ and $e_{l}$ are different types, then $n|\left(\left(m_{k}-m_{l}\right)^{2}+1\right)^{2}$
and hence $-1\equiv\left(m_{k}-m_{l}\right)^{2}$ is a square modulo
$p$ for any prime factor $p$ of $n$. Since $J$ always has at least two sides
of type $n$, if $n$ has a prime factor congruent to $3$ mod $4$,  all sides of $J$ must be of type $n$.
 The rest of the proof is divided into five cases:
\begin{enumerate}
    \item Suppose $n$ has a prime factor $p>5$ and $p$ congruent to 3 mod 4.
    Then all edges of $J$ must be of type $n$. Since $m_{k}-m_{l}$ is divisible by $p$ for any $k$ and $l$, the $J$-width of each vertex of $J$ is divisible by $p$ which contradicts Corollary \ref{cor:wt}. 
    
    \item Suppose $n$ has a prime factor $p>5$ and $p \equiv 1 \mod 4$ (so $p\ge 13$).
    Let $a$ be a square root of $-1$ mod $p$ between $1$ and $p$. It is easy to check that
$4\le a\le p-4$ and $\left|a-\left(p-a\right)\right|\ge3$. Indeed, $a^{2}+1\ge p\ge 13$ implies $a\ge4$. Similarly, $p-a\ge4$.
If $\left|a-(p-a)\right|<3$, then $a=\dfrac{p\pm 1}{2}$. This cannot
happen unless $p$ divides $5$, a contradiction. 

    Hence the $J$-width of the vertex between:
    \begin{enumerate}
        \item [(i)] two sides of type $1$ is congruent to $0$ or $2a$ or $-2a$ mod $p$. 
        \item [(ii)] a side of type $1$ and a side of type $n$ is congruent to $a$ or $p-a$ mod $p$. 
        \item [(iii)] two sides of type $n$ is divisible by $p$.
    \end{enumerate}
    Therefore, the weight of the vertex:
    \begin{enumerate}
        \item [(a)] between two sides of type 1 is at least $3$;
        \item [(b)] between a side of type 1 and a side of type $n$ is at least $4$;
        \item [(c)] between two sides of type $n$ is $1$ or at least $p$;
        \item[(d)] is $1$ only if it is between two type $n$ sides, so all ears of $J$ are $\triangle^{(n)}$ tiles, with tip between type $n$ sides.
    \end{enumerate}
    If $|J|=2$, (b) does not hold, so without loss of generality, we may assume $|J| \ge 3$, and we remove all the ears of $J$ which must be $\triangle^{(n)}$ tiles from (d) above. Consider any ear of the remaining polygon $J'$ with tip $v$. 
    \begin{itemize}
        \item If $v$ is bounded by two type $1$ sides in $J'$, as a vertex of $J$, its weight and $J$-width  is either $2$ or $3$ and it is bounded by either two type $n$ sides, or one type 1 and one type $n$ side in $J$ (by (d) above) which contradicts (b) and (c) above.
        \item If $v$ is bounded by two type $n$ sides in $J'$, then $J$ has a ear with tip between a type 1 and type $n$ side which contradicts (d).
        \item  If $v$ is bounded by a type 1 and a type $n$ side in $J'$, as a vertex in $J$, its weight is  $2$ (no ear can be attached to the type $n$ side by (d) above) and at least one of the sides is type $n$, which again contradicts (b) and (c) above.

    \end{itemize}

    

 \noindent   Therefore, the new polygon $J'$ has no ears (contradiction).
    
    \item Suppose $n$ is divisible by $15$.
    All the boundary sides of $J$ must be of type $n$. Hence the weight of each vertex is $1$ if it is the tip of a ear or at least $15$ otherwise. This contradicts Lemma \ref{lem:polygon} (b).

    \item Suppose $n=3^k$ with $k>2$.
    All edges are type $n$, and by Lemma \ref{lem:congruences} (2) the weight of each vertex is 1 or at least 9, which contradicts Lemma \ref{lem:polygon} (b).

    \item Suppose $n=5^k$ with $k>2$.
    By Lemma \ref{lem:congruences} (2),  if both $e_{k}$ and $e_{l}$ are type $n$, then $m_{k}-m_{l}$ is divisible by 25.
    If $e_{k}$ and $e_{l}$ are of different types, then $m_{k}-m_{l}$ is a square root of $-1$
modulo 25. 
    Then the $J$-width of the vertex between:

    \begin{itemize}
        \item two sides of type $1$ is congruent to 0 or 11 or 14 mod 25.
        \item two sides of different types is congruent to 7 or 18. 
        \item two sides of type $n$ is divisible by 25.
    \end{itemize}
    Hence the weight of the vertex between:
    \begin{itemize}
        \item two sides of type 1 is at least $11$.
        \item a side of type 1 and a side of type $n$ is at least $7$.
        \item two sides of type $n$ is 1 or at least 25.
    \end{itemize}
    This also contradicts Lemma \ref{lem:polygon} (b).

\end{enumerate}
This completes the proof of the proposition.






\end{proof}

It is tempting to think that the above can be strengthened to the remaining cases to show non-arithmeticity for all $n$, as was first suspected by the authors. However, it turns out that Theorem \ref{thm:nonarith} (i) is completed by the following somewhat surprising result, especially in view of part (ii) of the theorem:

\begin{prop}\label{prop:3,5,9,25}
There exists arithmetic $\SSS(1,n)$ jigsaw groups for $n\in \{3,5,9,25\}$ and these arise only as subgroups of the star jigsaw group when $n=3$ or $9$, and the diamond jigsaw group when $n=5$ or $25$.
\end{prop}

\begin{proof}
 The cases $n=3$ and $9$ are similar, while the cases $n=5$ and $25$ are similar.
When $n=3$, \cite[proposition 7.4]{LTV}  states that the star jigsaw group $\Gamma_J$ is arithmetic, we provide more details so the same argument can be applied to the other  cases. We have
 $$\Gamma_J=\langle \iota_0, \ldots, \iota_5 ~|~ \iota_0^2=\cdots =\iota_5^2=Id \rangle $$
 where $\iota_k$ is the $\pi$-involution about the marked point $x_k$ on the $k$th side of $J$, which can be computed explicitly. Then 
 $$H_J=\langle \gamma_1, \ldots, \gamma_5 ~|~  \gamma_k=\iota_0\iota_k, \quad k=1, \ldots, 5 \rangle.$$
 By \cite[Cor 3.2]{HLM}, to show the invariant trace field is $\QQ$, it suffices to check that $\tr g \in \ZZ$ for all $g \in H_J^{(SQ)}:=\langle \gamma_1^2, \ldots, \gamma_5^2 \rangle$ and by \cite{Hor}, it suffices to check this for the $2^5-1=31$ words of the form 
 $$\gamma_{\nu_1}^2\gamma_{\nu_2}^2\ldots \gamma_{\nu_k}^2, \quad \hbox{where} \quad \nu_1<\nu_2 \cdots <\nu_k; \quad k \le n.$$
 To check that $\tr g$ is an algebraic integer for all $g \in H_J$, it suffices to check that $(\tr g)^2 \in \ZZ$ for the words $g$ of the form
 $$\gamma_{\nu_1}\gamma_{\nu_2}\ldots \gamma_{\nu_k}, \quad \hbox{where} \quad \nu_1<\nu_2 \cdots <\nu_k; \quad k \le n.$$
 The two conditions imply $\tr \gamma^2 \in \ZZ$ for all $\gamma \in H_J$ and hence arithmeticity of $H_J$. This is easily verified computationally, and we do the same for $n=9$ to verify that the star jigsaw in that case is also arithmetic (see Appendix A).
 
 It turns out that every arithmetic $\SSS(1,3)$ jigsaw group is a subgroup of the star jigsaw group (proposition 7.6 \cite{LTV}). Exactly the same argument works for the case $n=9$, the key observations are that if the jigsaw is arithmetic, then it cannot have any type 1 sides and must   decompose into a union of star blocks.
 
 For $n=5$ or $25$, $H_J$ for the diamond jigsaw is a free group of rank $3$, and a direct computation as above checking the traces of the relevant $7+7$ words shows that the diamond jigsaw group  is arithmetic (see Appendix A).
 
 Showing that all arithmetic examples arise as subgroups of the diamond jigsaw groups requires a bit more work, as for the $n=3,9$ cases with star jigsaws. Let $J$ be a $\SSS(1,n)$ jigsaw with $|J|=N$, $s_i=[v_i,v_{i+1}]$, where $i=0, \dots, N+1$ be the sides of $J$. Define $$m_{ij}:=\sum_{k=i+1}^{j} \hbox{JW}(v_k)=m_i-m_j,$$ the sum of the $J$-widths of all  vertices between $s_i$ and $s_j$, where indices are taken modulo $N+2$, and $m_j$ and $m_i$ are the endpoints of the vertical sides $e_i$ and $e_j$ of $\QQQ_J$ as defined in Lemma \ref{lem:QTV}. So $m_{ij}+m_{ji}=L$ where  $L=\sum_{k=0}^{N+1}\hbox{JW}(v_k)$ is the length of the fundamental interval. The conditions for $J$ to be admissible using the congruences from Lemma \ref{lem:congruences} easily translates to:
 \begin{lem}\label{lem:525}
 For $n=5,25$, the $\SSS(1,n)$ jigsaw $J$ is admissible if and only if the following conditions hold:
 \begin{enumerate}
    \item If $s_{k}$ and $s_{l}$ are type $n$, then $m_{kl}$ is
divisible by 5.
    \item If $s_{k}$ and $s_{l}$ are different types, then $m_{kl}\equiv  2 ~~\hbox{or}~~ 3 \mod 5$.
\item If $s_{k}$ and $s_{l}$ are type $1$, then $m_{kl}\equiv 0,1 ~~\hbox{or}~~-1 \mod 5$.
\item $L$ is divisible by $5$.
\end{enumerate}
\end{lem}
Suppose $J$ is admissible. It follows that if $E$  is an ear of $J$ of type $n$, then the tip of $E$ must be between two type $n$ sides and also, $|J| \neq 1,3$ and if $|J|=2$, then $J$ is a diamond. Hence we may assume $|J|\ge 4$. Let $\triangle_1, \ldots, \triangle_a$ be the ears of $J$ and let $$J':=J\setminus\cup_{i=1}^{a}\triangle_i.$$ 
Let $E$ be an ear of $J'$ with tip $v$. The weight of $v$ in $J$ is either $2$ or $3$, depending on whether $E$ has two or three ears adjacent to it which are not in $J'$. 
 Figure \ref{fig:table_of_cases} gives all possible configurations   where $E$ is the shaded triangle, sides of $J$ are bold lines, and the two sides which do not satisfy the congruence relations from Lemma \ref{lem:525} are marked by the red crosses in the figures on the left. It follows that the only possible admissible configurations (on the right) are that $E$ has exactly one ear $E'$ adjacent to it of a different type and $E \cup E'$ is a diamond $D$.

\begin{figure}
     \centering
     \includegraphics[width=\textwidth]{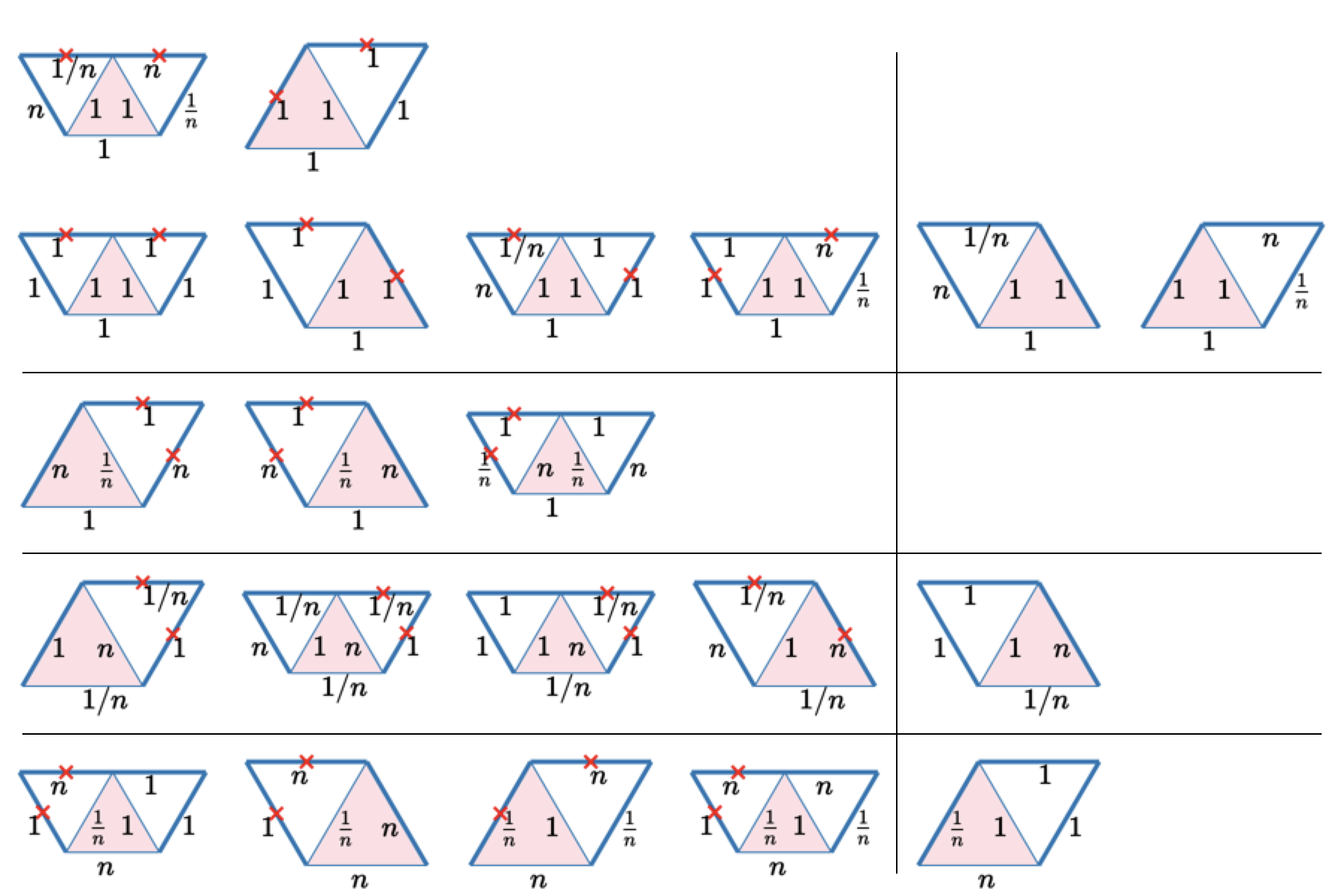}
    \caption{Enumeration of non-admissible (left column) and admissible (right column) cases for checking arithmeticity of $\SSS(1,n)$-jigsaws ($n=5$ or $25$). The shaded triangle is the ear of $J'$. }
    \label{fig:table_of_cases}
\end{figure}

In particular, $J$ decomposes into a union of $D$  with a jigsaw $J_2$. As before, Lemma \ref{lem:525}(4)  implies $|J_2| \neq 1,3$ and if $|J_2|=2$, then $J_2$ is a diamond, hence admissible. Otherwise, $|J_2| \ge 4$ and it is straightforward to check that $J_2$ satisfies the conditions of Lemma \ref{lem:525} and so is admissible. Indeed, denote $e_0$ to be the common edge of the diamond $D$ and $J_2$. To check the admissibility of $J_2$, it suffices to check the congruence of the sum of $J$-widths of vertices between $e_0$ and any other edge $e$ of $J_2$. This can be done by using other edges of $D$, given the admissibility of $J$. Proceeding inductively, we obtain a sequence 
$$J=J_1\supset J_2 \ldots, \supset J_{N/2}=D$$
of admissible $\SSS(1,n)$ jigsaws where $J_{N/2}$ is the diamond jigsaw $D$ and $J_k$, $2 \le k \le N/2$, is obtained from $J_{k-1}$ by peeling off a diamond. In particular, $N$ is even. We now show by reversing the process that  $J_k \subset \QQQ_D$ and $\Gamma_{J_k}$ is a subgroup of the diamond jigsaw group $\Gamma_D$ for $k=1, \ldots, N/2$. The statement is certainly true for $k=N/2$ since $J_{N/2}=D$, and $D \subset \QQQ_D$. Assume that $J_{i+1} \subset \QQQ_D$ and $J_i$ is obtained from $J_{i+1}$ by attaching a diamond, where both are admissible. If the diamond is attached along a type $n$ side, there is only one way to do it, which is compatible with $\QQQ_D$, so   $J_i \subset \QQQ_D$. If the diamond is attached along a type 1 side to $J_{i+1}$, there are two ways to do it (right side of first row of Figure \ref{fig:table_of_cases}), but at most one of these ways gives an admissible jigsaw $J_i$. Indeed, this can be deduced by comparing $J$-width between a side $e$ of $J_{i+1}$ of type $n$, which always exists, and sides of the diamond $D$ in the two ways of gluing. On the other hand, attaching a diamond to $J_{i+1}$ along that side in a way  compatible with $\QQQ_D$ produces an admissible jigsaw since the group is arithmetic. Hence $J_i$ is obtained by attaching a diamond in the correct way, compatible with $\QQQ_D$  and $J_i \subset \QQQ_D$. Proceeding inductively, we conclude that $J=J_1 \subset \QQQ_D$, which implies that $\Gamma_J < \Gamma_D$ as required. This completes the proof of Proposition \ref{prop:3,5,9,25}.

\end{proof}

\begin{proof}(Theorem \ref{thm:nonarith}) Part(i) follows from propositions \ref{prop:even}, \ref{prop:othercases} and \ref{prop:3,5,9,25}. For part(ii) , the second part of Corollary \ref{cor:wt} applied to the Weierstrass group implies that $n$ divides $4$. A direct verification shows that $n=1,2$ and $4$  gives arithmetic groups.
\end{proof}

\medskip

\noindent {\it Remark:} We are not sure if all $\SSS(1,n_2,\ldots, n_s)$-jigsaw groups are not arithmetic if $s\ge 3$ although in many cases, we can rule out arithmeticity. For example, it is easy to adapt the proof of proposition \ref{prop:even} to show that such groups cannot be arithmetic if at least one of $n_k$ is even.

\subsection{Arithmetic Diamond $\SSS(m,n)$ jigsaws}\label{ss:mnjigsaws}
We drop the condition that $\triangle^{(1)}$ is one of the tiles in an integral jigsaw set in this subsection and consider the more general jigsaw sets $\SSS(m,n):=\{\triangle^{(m)}, \triangle^{(n)} \}$, $1\le m \le n$. Here the diamond jigsaw consists of one tile of each type glued along the type 1 side (with suitable interpretation if $m=n$). We have:

\begin{thm}\label{thm:mn} Let $\mathcal{M}^{ref}{(0;2,2,2,2;1)}$ be the moduli space of hyperbolic surfaces with signature $(0;2,2,2,2;1)$ admitting a reflection symmetry where the fixed point locus of the reflection does not contain any of the order 2 points. There are exactly eight  arithmetic surfaces in $\mathcal{M}^{ref}{(0;2,2,2,2;1)}$. They correspond to the diamond $\SSS(m,n)$ jigsaw groups for the pairs $$(1,1), (1,5), (2,2), (2,6), (4,4), (4,8), (8,12), (16,20).$$ On the other hand, there are infinitely many pairs $(m,n)\in \NN\times \NN$ for which the diamond $\SSS(m,n)$-jigsaw group is arithmetic. They arise from the action of the mapping class group on the eight basic pairs above. More precisely, any arithmetic pair $(m,n)$ arises from one of the basic pairs above under the following infinite cyclic group action:
\begin{equation*}\label{eqn:mngenerator}
      (m,n) \quad \xrightarrow{R} \quad (n, (l-2)n-8-m), \quad \hbox{where} \quad l=(m+n+4)^2/mn\in \NN.
  \end{equation*}

\end{thm}

\begin{proof}
 Suppose that $S$ is such an arithmetic surface with holonomy $\Gamma$. We may take as the fundamental domain the ideal polygon $J$ with vertex set $\{\infty, -1,0,1\}$,  such that $\Gamma$  is generated by the involutions $ \iota_1, \iota_2, \iota_3, \iota_4$,  about the order two points $x_1,x_2,x_3,x_4$  on the sides of $J$ respectively, and such that the fixed point set of the reflection is the geodesic $[-1,1]$. Due to the symmetry,  $J$, with the marked points on the sides can be decomposed to the marked triangles $(\infty, -1,0)=\triangle(1/m,m,1)$ and $(\infty, 0, 1)=\triangle(1,1/n,n)$ where $m,n \in \RR_+$, and $x_1,x_2,x_3,x_4$ correspond to the marked points on the sides of these triangles, see Figure \ref{fig:mnA}, where $m=2$ and $n=5$, and the reflection axis is colored red. 
 
 \begin{figure}
     \centering
     \begin{subfigure}[b]{0.352\textwidth}
         \centering
         \includegraphics[width=\textwidth]{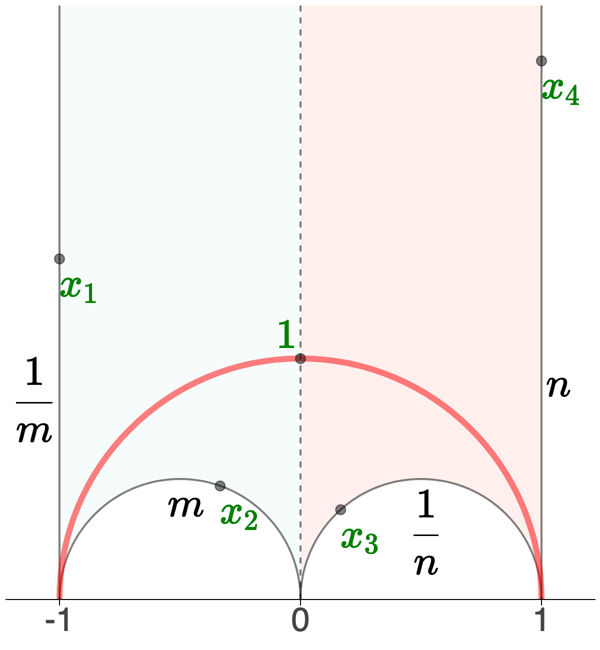}
         \caption{ $J$, with shaded tiles }
         \label{fig:mnA}
     \end{subfigure}
     \begin{subfigure}[b]{0.638\textwidth}
         \centering
         \includegraphics[width=\textwidth]{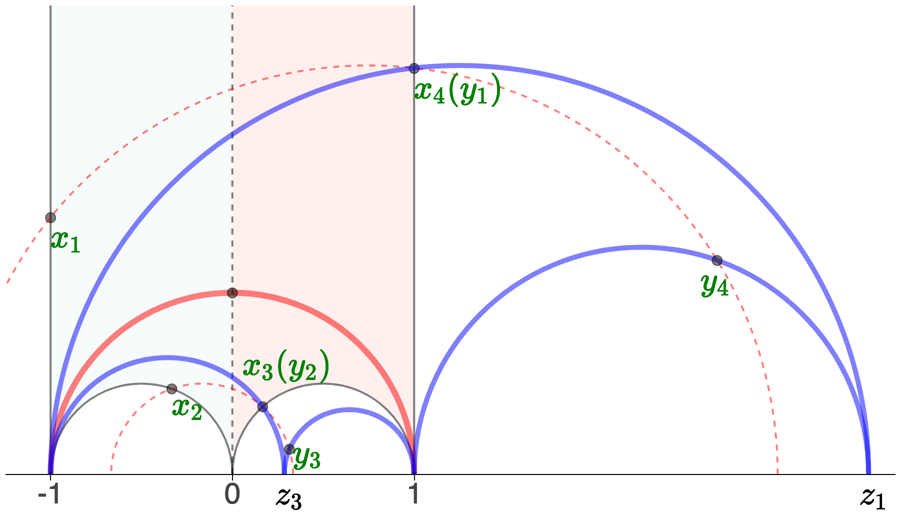}
         \caption{$J'$ bounded by blue geodesics }
         \label{fig:mnB}
     \end{subfigure}
        \caption{The diamond $\SSS(2,5)$ jigsaw $J$ and its image $J'$ under the action of $R$. }
        \label{fig:mn}
\end{figure}
 
 We first note that since $\Gamma$ is arithmetic, the cusp set is $\QQ \cup \{\infty\}$ which implies $m,n \in \QQ$ since $0,-1$, $-1-m$, $1$ and $1+n$ are cusps. Applying (\ref{eqn:tracesquare}) to $\iota_1(\iota_1\iota_2\iota_1)=\iota_2\iota_1$, we get 
  $$(\tr (\iota_2\iota_1))^2 = \dfrac{(m^2-2m)^2}{m^2}=(m-2)^2\in \NN \quad \Longrightarrow \quad m \in \NN.$$
 Similarly, by considering  $\iota_3\iota_4$, we get that $n \in \NN$.   Hence $J$ is an $\SSS(m,n)$ diamond jigsaw, where $m,n \in \NN$. Applying the criteria for arithmeticity, Proposition \ref{prop:mn} lists  the pairs $(m,n)$, $1 \le m \le n \le 10,000$ giving rise to arithmetic groups, which includes the eight listed pairs, which motivates the following analysis. 
 
 We now explain how any arithmetic  $(m,n)$ pair corresponds to a hyperbolic surface isometric to one of these eight surfaces. 
 Applying (\ref{eqn:tracesquare}) to $\iota_1\iota_3$, we get 
  \begin{equation}
      (\tr (\iota_1\iota_3))^2 = \dfrac{(4+m+n)^2}{mn}:=l\in \NN.
  \end{equation}
  Furthermore, since $\iota_1\iota_3$ is hyperbolic, $l \ge 5$.
  The subgroup of the mapping class group for the surface preserving the reflection symmetry is infinite cyclic, say with generator $R$.  We now describe how $J$  and the pair $(m,n)$ transforms under the action of $R$. Let $z_1$ and $z_3 \in \RR$ be the end points of the geodesics from $-1$ through $x_4$ and $x_3$ respectively, see figure \ref{fig:mnB} . Then it is easy to see that $J'$ with vertices at $z_1, -1, z_3, 1$ is a fundamental domain for $S$, and the fixed point set of the reflection remains $[-1,1]$. $J'$ arises from the new marking of the surface under the transformation $R$. The marked points (fixed points of the involutions) are at $y_1,y_2,y_3$ and $y_4$ where $y_1=x_4$, $y_2=x_3$, $y_4=\iota_4(x_1)$ is the intersection of $[1,z_1]$ and the geodesic through $x_1$ and $x_4$, and $y_3=\iota_3(x_2)$ is the intersection of $[1,z_3]$ and the geodesic through $x_2$ and $x_3$. In short,
  $$J\xrightarrow{R}J'$$
  and $J'$  again decomposes to two integral tiles of type $p$ and $q$. Since $J$ and $J'$ share the same reflection axis and fixed points $x_3$ and $x_4$, we have $p=n$, and writing $q=m'$, we see that $m,m'$ are the roots of the quadratic  polynomial 
  \begin{equation}
      (4+x+n)^2=lxn \quad \Longleftrightarrow \quad x^2+(8+2n-nl)x+(n+4)^2=0
  \end{equation}
  It follows that 
  \begin{equation}\label{eqn:mnrelation}
      m+m'=(l-2)n-8, \qquad mm'=(n+4)^2,
  \end{equation}
  and \begin{equation}\label{eqn:mngenerator}
      (m,n) ~~ \xrightarrow{R} ~~ (n, (l-2)n-8-m), \quad (m,n)~~ \xrightarrow{R^{-1}}~~ ((l-2)m-8-n, m).
  \end{equation}
  Next, we claim that if $n\ge 25$, then $\min (m,m')<n$. Suppose not. Then by (\ref{eqn:mnrelation}), $m=(n+4)^2/m'\le(n+4)^2/n\le n+8$, and similarly, $m'\le n+8$. This implies $m+m'\le 2n+16$, but $m+m'=(l-2)n-8>3n-8$ since $l \ge 5$ which gives a contradiction if $n \ge 25$. It follows that an arithmetic pair is equivalent to a pair $(m,n)$  where $n\le 25$, and again by (\ref{eqn:mnrelation}), $m\le 29$. This occurs as one of the pairs in the list of Proposition \ref{prop:mn}, which is equivalent to one of the eight pairs, noting that successive pairs $(m,n)$ in each row of Table \ref{tab:diamonds} are given by the action in (\ref{eqn:mngenerator}). 
  
  Alternatively, we note that $(m,n)$ and $(n,m)$ are equivalent as they give rise to the same surface, so if we take the pair $(m,n)$ in an equivalence class with the smallest entries and $m \le n$, then either $m=n$ or $(m,n)$ is the image of $(n,m)$ under $R$. In the first case, we have $(4+m)^2/m^2 \in \NN$ which implies $m=1,2$ or $4$, all of which gives arithmetic surfaces by Theorem \ref{thm:nonarith}(2). In the second case, interchanging $m$ and $n$ in (\ref{eqn:mnrelation}), we get 
  $$n^2=(m+4)^2, \quad 2n=(l-2)m \Longrightarrow 2(4+m)=(l-2)m \Longrightarrow 2+\frac{8}{m} \in \NN$$
  from which we get $m=1,2,4$ or $8$ with corresponding $n=5,6,8$ or $16$ respectively. Again one verifies that all five pairs $(1,5), (2,6), (4,8)$ and $(8,12)$ gives rise to arithmetic surfaces.
  
  Finally, we note that $l=(m+n+4)^2/mn$ is an invariant of the surface, so the eight pairs give distinct surfaces, which completes the proof.
\end{proof}
\noindent{\it Remark:} For each $l=5,6,8,9,12,16,20,36$, writing the arithmetic pairs in each orbit as $(m_i,m_{i+1})$, $i \in \ZZ$, (with minimal pair $(m_1,m_2)_l$) the linear recursion (\ref{eqn:mnrelation}) can be expressed as 
\begin{equation}
\left(\begin{array}{c}
m_{k+2}\\
m_{k+1}\\
1
\end{array}\right)
=\left(\begin{array}{ccc}
l-2 & -1 &-8\\
1 & 0 & 0\\
0 & 0 & 1
\end{array}\right)^k
\left(\begin{array}{c}
m_{2}\\
m_{1}\\
1
\end{array}\right)
\end{equation}

For example, for $l=5$, $(m_1,m_2)_5=(16,20)$, for $l=9$, $(m_1,m_2)_9=(4,4)$ (see Table \ref{tab:diamonds}).


\section{Pseudomodular groups from integral jigsaws}\label{s:PMG}
The main aim of this section is to prove the following:

\begin{thm} \label{thm:infinitepmg}
	For any integral jigsaw set  $\SSS(1,n_2,\ldots, n_s)$, $1<n_2<\cdots <n_s$, there exists infinitely many commensurabilty classes of  $\SSS(1,n_2,\ldots, n_s)$  pseudomodular jigsaw groups. 

\end{thm}

The main tool we use is the Killer intervals introduced by Long and Reid in \cite{LR}.

\subsection{Killer Intervals}\label{ss:KI}
We have the following result which is a key idea from \cite{LR}.

\begin{prop}\cite{LR}, \cite[Prop 3.1]{LTV}. Suppose $\Gamma<\PGLtwoQ$ and
\[g = \left(
	\begin{array}{cc}
	\alpha & \beta \\
	\gamma & \delta \\
	\end{array}
	\right) \in \Gamma, \quad \hbox{where}\qquad \alpha, \beta, \gamma, \delta \in \ZZ, \quad  \gcd(\alpha, \beta, \gamma, \delta)=1.\]
	
	Suppose  further that $\gcd(\alpha, \gamma)=k$ so $\alpha=k\alpha'$, $\gamma=k\gamma'$ where $\gcd (\alpha', \gamma')=1$. Then for any $\dfrac{p}{q} \in \left(\dfrac{\alpha'}{\gamma'} -\dfrac{1}{\gamma}~,~\dfrac{\alpha'}{\gamma'} +\dfrac{1}{\gamma}\right)$,  $g^{-1}\left(\dfrac{p}{q}\right)$ has strictly smaller denominator than $\dfrac{p}{q}$.
\end{prop}
Note that we chose a representative $g\in \Gamma$ with integer entries, the determinant is not necessarily one. The interval $(\dfrac{\alpha'}{\gamma'} -\dfrac{1}{\gamma},\dfrac{\alpha'}{\gamma'} +\dfrac{1}{ \gamma})$ is called the {\em Killer Interval} about the cusp $\dfrac{\alpha'}{\gamma'}$,  its {\em radius} is $\dfrac{1}{\gamma}$, $k$ is called the {\em contraction constant} for $\dfrac{\alpha'}{\gamma'}$ and in the case of integral jigsaw groups, does not depend on the choice of $g$.

\subsection{Farey blocks}\label{ss:Fareyblocks}
 We will construct the jigsaws judiciously so that it is easy to find a finite set of Killer intervals which cover the entire fundamental interval. Using the associated pseudo-euclidean algorithm, we deduce that the set of cusps is $\QQ \cup \{\infty\}$. For this we will introduce the idea of attaching  $n$-Farey blocks to the $\triangle ^{(n)}$ tiles. 
\begin{defn}
    An $r$-Farey block $B_r$, $r \ge 2$ is an ideal polygon/jigsaw with a distinguished side, assembled from $\triangle^{(1)}$ tiles inductively, starting by placing one tile on the triangle $(0/1,1/2,1/1)$ and adding more tiles until the vertices of $B_r$ are precisely the set $$\{p/q ~|~ 0 \le p/q\in \QQ \le 1, \gcd(p,q)=1, ~~~1 \le q \le r \}.$$ See Figure 5 for $B_3$ and $B_5$.  The side $[0,1]$ (colored red) is the distinguished side of $B_r$.
\end{defn}

\begin{figure}
     \centering
     \begin{subfigure}[b]{0.48\textwidth}
         \centering
         \includegraphics[width=\textwidth]{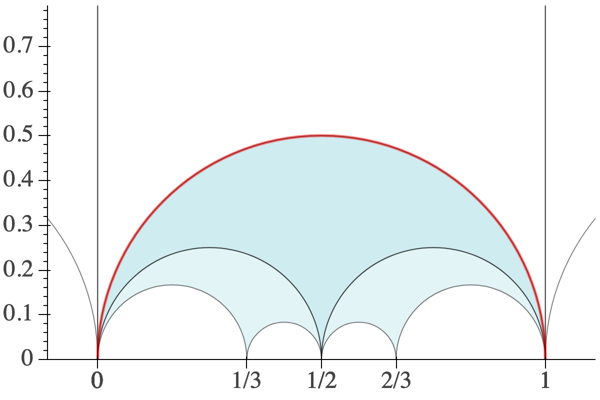}
         \caption{The 3-Farey block $B_3$}
         \label{fig:y equals x}
     \end{subfigure}
     \begin{subfigure}[b]{0.48\textwidth}
         \centering
         \includegraphics[width=\textwidth]{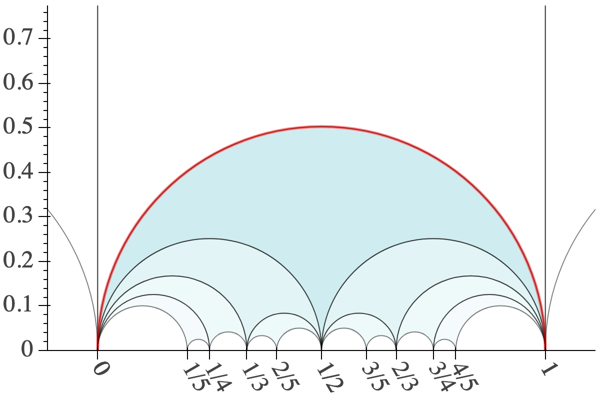}
         \caption{The 5-Farey block $B_5$}
         \label{fig:three sin x}
     \end{subfigure}
        \caption{ The 3-Farey block and 5-Farey block }
        \label{fig:three graphs}
\end{figure}

\begin{defn}
    An $\SSS(1,n)$ jigsaw $J$ is {\it good} if it has an $n$-Farey block attached  to every $\triangle^{(n)}$ tile in $J$, along the $1$-side of the tile, to the distinguished side of the block. Similarly, we can extend the definition to $\SSS(1, n_2, \ldots, n_s)$ jigsaws where every $\triangle^{(n_k)}$-tile, $k\ge 2$,  has a $n_k$-Farey block attached to it.
\end{defn}
\medskip
 In particular, for $\SSS(1,n)$, if there is only one $\triangle^{(n)}$-tile, we only need to attach a $n$-Farey block to it, and attaching any other number of $\triangle^{(1)}$ tiles  will give a good $\SSS(1,n)$-jigsaw $J$. 
 
 \subsection{Proof of Theorem \ref{thm:infinitepmg}} We first state and prove the result for $\SSS(1,n)$:

\begin{prop}
 Fix the jigsaw set $\SSS(1,n)$ where $n \ge2$.
 \begin{itemize}
    \item[(a)]For any good $\SSS(1,n)$ jigsaw $J$, $\Gamma_J$ is pseudomodular.
    \item [(b)] There are infinitely many non-commensurable pseudomodular $\SSS(1,n)$ jigsaw groups.
\end{itemize}
    
\end{prop}

\begin{proof}
We first show (a) by showing that a fundamental interval $[0,L]$ can be covered by Killer intervals. By Lemma \ref{lem:QTV}(1) the vertical edges of $\TTT_J$ have endpoints in $\ZZ$, with the real endpoints of successive edges differing by either $1$ or $n$, the latter occurring only if the edges are sides of a $\triangle^{(n)}$ where the vertical sides of this triangle are type $n$. Furthermore, by  \cite[Proposition 4.6]{LTV},  the radius of the Killer intervals about each of these endpoints is $1$. Hence the Killer intervals about these endpoints cover most of the fundamental interval, the only thing we need to worry about are the intervals between $r$ and $r+n$ where $(\infty, r, r+n)$ is a $\triangle^{(n)}$ in $\TTT_J$. However, since $J$ is good, we know that there is a $n$-Farey block attached to the side $[r,r+n]$. Hence the rationals $r+\frac{np}{q}$ where $0\le \frac{p}{q} \le 1$, $\gcd(p,q)=1$, $1\le q\le n$ are cusps of $\Gamma_J$. If $(r+\frac{p}{q}, r+\frac{p+u}{q+v}, r+\frac{u}{v})$ is a Farey triple, with $0\le \frac{p}{q}<\frac{p+u}{q+v}<\frac{u}{v}\le 1$ and $q+v\le n$, then $(r+\frac{np}{q}, r+\frac{np+nu}{q+v}, r+\frac{nu}{v})$ is a $\triangle^{(1)}$ tile of $\TTT_J$. On the other hand, there is another lift of this triangle which is $(t+1,\infty, t)$ for some $t\in \ZZ$. A direct calculation as in \cite{LTV} shows that the map from $(t+1,\infty, t)$ to $(r+\frac{np}{q}, r+\frac{np+nu}{q+v}, r+\frac{nu}{v})$ is of the form 
$\left(
			\begin{array}{cc}
			n(u+p)+r(q+v) & * \\
			q+v & * \\
			\end{array}
			\right) $ where all entries are integers, which gives that the Killer interval about $r+\frac{np+nu}{q+v}$ has radius $\frac{1}{q+v}$. If $\frac{p}{q}$ and $\frac{u}{v}$ are successive vertices of the $n$ Farey block, we have 
\begin{equation}
    \left|(r+\frac{nu}{v})-(r+\frac{np}{q})\right|=\left|\frac{n}{vq}\right|
    <\left|\frac{1}{v}+\frac{1}{q}\right|
\end{equation}
since $q+v> n$ by definition of the $n$-Farey block.
It follows that the Killer intervals about the cusps $r+\frac{np}{q}$ cover the interval $[r, r+n]$, and hence that the fundamental interval for $\Gamma_J$ can be covered by Killer Intervals. As in \cite{LR}, this implies that the cusp set for $\Gamma_J$ is $\QQ \cup \{\infty\}$.
On the other hand, by theorem \ref{thm:nonarith}, if $n\neq 3,5,9,25$, $\Gamma_J$ cannot be arithmetic and if $n=3,5,9$ or $25$, it is clear that good jigsaw groups cannot be subgroups of star jigsaw groups or diamond jigsaw groups and so also cannot be arithmetic. Hence, all good $\SSS(1,n)$ jigsaw groups are pseudomodular.

\medskip
To show (b), let $J$ be a good $\SSS(1,n)$-jigsaw of signature $sgn(J)=(M,1)$, so $J$ is an $(M+3)$-gon. It is peudomodular by (a), and has two adjacent sides of type $n$,  all other sides are type $1$.
By Margulis's result from \cite{Mar} about non-arithmetic lattices, the commensurator $\hbox{Comm}(\Gamma_J)$ is discrete and is the maximal element in the commensurability class of $\Gamma_J$. We show that $\Gamma_J=\hbox{Comm}(\Gamma_J)$. The proof is similar to that used in  \cite[Section 8]{LTV}. We will look at the maximal horocycle of $\HH/\Gamma_J$ and show that there are two points of self-tangency which divides the horocyle into a long segment and a short segment. 

Let $C$ be a horocyle of $\HH/\Gamma_J$ and for a cusp $v=g(\infty)$ of $\Gamma_J$, where $g \in \Gamma_J$, let $\tilde C_{v}$ be the component of the lift of $C$  tangent to $v$. In particular,  $\tilde C_{\infty}$ is a horizontal line, and $\tilde C_{v}=g(\tilde C_{\infty})$. Let $[\infty, r]$ and $[\infty, r+n]$ be successive vertical sides of $\QQQ_J$ of type $n$, where $r \in \ZZ$. Without loss of generality, by shifting everything to the left by $r$, we may assume that $r=0$. It is convenient to consider  the fundamental domain for $\Gamma_J$ having these two sides as part of the boundary, with vertices $v_0=\infty, v_1=0<v_2<\cdots <v_{M+2}=n<v_{M+3}=\infty=v_0$. Note that $v_k\in \QQ$ for $k=1, \ldots, M+2$ and apart from the triangle $(\infty, 0,n)$, all the other triangles in the fundamental domain are $\triangle^{(1)}$ tiles, whose positions differ from those obtained from  the Farey triangulation of $\HH$ by a dilation of factor $n$. Hence, $v_k=np_k/q_k$, where $\gcd(p_k,q_k)=1$. Now suppose that $\tilde C_{\infty}=\{x+ih ~|~ x \in \RR\}$. By a straightforward computation, $\tilde C_{np_k/q_k}$ is a circle tangent to $np_k/q_k$   of diameter $n/hq_k^2$. In particular, the circles $\tilde C_{np_k/q_k}$ for $k=1, \dots, M+2$ are mutually disjoint if $h>1$, and if $h=1$ two such circles are mutually tangent if and only if $p_k/q_k$ and $p_l/q_l$ are Farey neighbors. On the other hand, these circles are all disjoint from $\tilde C_{\infty}$ if $h >\sqrt{n}$, and if $h=\sqrt{n}$, then only $\tilde C_{v_1}$ and $\tilde C_{v_{M+2}}$ are  tangent to $\tilde C_{\infty}$. It follows that the maximal horocycle on $\Gamma_J$ lifts to the horizontal line $\{x+i\sqrt{n} ~|~ x \in \RR\}$, it is self-tangential at the two marked points on the two type $n$ sides of $J$ which divide it into a short and long segment, if $3M+2>n$, which is clearly true if $J$ is good. Using the same argument as in \cite[Section 8]{LTV}, we may conclude that $Comm(\Gamma_J)=\Gamma_J$. Hence, for different $M$,  good $\SSS(1,n)$ jigsaws of signature $(M,1)$ are in  different commensurability classes, and there are infinitely many such commensurability classes.
\end{proof}

\begin{proof} (Theorem \ref{thm:infinitepmg}) We have proven the theorem in the case of $\SSS(1,n)$-jigsaws. Now consider a good $\SSS(1,n_2, \ldots, n_s)$ jigsaw $J$ of signature $(M,1, \ldots, 1)$. 
 Note that for $k=2, \ldots, s$, all the $\triangle^{(n_k)}$ tiles are ears of $J$ whose tips are between the two sides of type $n_k$. Using the congruence arguments from Section \ref{s:nonarith}, it is not difficult to show again that these jigsaws cannot be arithmetic, we leave the details to the reader. On the other hand, the same argument as before with the Killer intervals implies that the cusp set is $\QQ \cup \{\infty\}$. Hence all good $\SSS(1,n_2,\ldots, n_s)$ jigsaw groups of signature $(M,1,\ldots, 1)$ are pseudomodular. We next show that $\hbox{Comm}(\Gamma_J)=\Gamma_J$ for these jigsaw groups again using the maximal horocycle argument. We claim that   the lift $\tilde{C}_\infty$ of the maximal horocycle $C$ tangent to $\infty$ is the horizontal line $\{x+i\sqrt{n_s} ~|~x \in \RR\}$ and that $C$ has two points of self tangency, which are exactly the marked points on the type $n_s$ sides of the $\triangle^{(n_s)}$ tile. Consider again the fundamental domain where the two type $n_s$ sides are vertical and at distance $n_s$ apart. Essentially the same argument as before gives that the lifts of this horocycle about vertices of $J$ which are not tips lying between two type $n_k$ ($k\ge 2$) edges are disjoint and do not intersect $\tilde{C}_{\infty}$, except for $\tilde{C}_{v_1}$
and $\tilde{C}_{v_{M+s}}$ which are tangent it.

 We need to show that the lifts of this horocycle to the tips of the type $n_j$ and $n_k$ tiles, where $2\le j<k\le s$ are disjoint. For this consider the fundamental domain for $J$ which is bounded by two type $n_k$ sides, so now $\infty$ is the tip of the type $n_k$ tile, and again by translation if necessary, assume that the real endpoints of these sides are at $0$ and $n_k$. The tip of the $n_j$ tile is at a rational number $n_kp/q$ where $0 \le p/q \le 1$ and $q>n_k$ since $J$ is good. A simple computation gives that the lift of the horocycle tangent to the tip of type $n_j$ at $n_kp/q$ has diameter $\dfrac{n_jn_k}{\sqrt{n_s}q^2}<\sqrt{n_s}$, so is disjoint from the lift tangent to the tip of type $n_k$ (which is at $\infty$). Note that if $n_j=1$ and $k<s$, the inequality still holds regardless of $q$. Hence, this is indeed the maximal horocycle which has exactly two self tangency points dividing it into a long and a short segment. As before this implies  $\hbox{Comm}(\Gamma_J)=\Gamma_J$ from which we obtain infinitely many non-commensurable pseudomodular groups from  good $\SSS(1,n_2, \ldots, n_s)$-jigsaws of signature $(M,1,\ldots,1)$ by using different $M$.
\end{proof}

\section{Cutting sequences, Pseudo-Euclidean algorithms and generalized continued fractions}\label{s:PSAlgorithms}
    We describe in this section some algorithms for obtaining cutting sequences, generalized continued fractions and a pseudo-euclidean algorithm (for the $\SSS(1,2)$ diamond jigsaw group) associated to the Weierstrass groups and pseudomodular jigsaw groups.

\subsection{Algorithms for Weierstrass groups}\label{ss:algoWeier}
 Let $\Gamma:=\Gamma(k_1,k_2,k_3)$ and $\TTT:=\TTT_{(k_1,k_2,k_3)}$, to simplify notation. Recall that $\triangle_0=(\infty, -1,0)$ is a fundamental domain for $\Gamma$ and $\TTT=\Gamma(\triangle_0)$ is the associated triangulation of $\HH$, see Figure \ref{fig:Tandtriangulation}. Coloring the sides $s_1=[\infty, -1]$, $s_2=[-1,0]$ and $s_3=[0, \infty]$ of $\triangle_0$ by $1,2$ and $3$ respectively induces a coloring of the sides of $\TTT$.  Pick an interior point $p_0\in \triangle_0$  and let $\alpha \in \RR\setminus\{-1,0\}$. The geodesic from $p_0$ to $\alpha$ cuts through the edges of $\TTT$ with an induced cutting sequence $CS({\alpha})=:c_1c_2\ldots c_k\ldots$, where $c_k \in \{1,2,3\}$, $c_k \neq c_{k+1}$ for all $k \in \NN$.  We describe a couple of algorithms to obtain $CS({\alpha})$.

\begin{algo}\label{alg:naive} (Naive algorithm)
Let $\iota_1, \iota_2, \iota_3$ be the involutions on the marked points $x_1,x_2$ and $x_3$ of the sides $s_1,s_2$ and $s_3$ of $\triangle_0$ and let $I_1, I_2, I_3 \subset \RR$ be the corresponding intervals in $\RR$ where 
\begin{equation}
I_1=\{x \in \RR ~|~ x< -1\}, \quad 
I_2=\{x \in \RR ~|~ -1< x< 0\}, \quad I_3=\{x \in \RR ~|~ x>0\}. 
\end{equation}
Define $\alpha_0:=\alpha$, and for $k \ge 0$, if $\alpha_k\in I_{b_{k+1}}$, where $b_{k+1} \in \{1,2,3\}$, define $\alpha_{k+1}=\iota_{b_{k+1}}(\alpha_k)$. The algorithm terminates if $\alpha_n\in \{\infty, -1,0\}$ for some $n \in \NN$, otherwise we obtain an infinite sequence $\alpha_0, \alpha_1, \ldots, \alpha_{k}, \ldots$ where
\begin{equation}
    \alpha_{n+1}=\iota_{b_{n+1}}\iota_{b_{n}}\ldots \iota_{b_1}(\alpha_0), \quad \alpha_{0}=\iota_{b_{1}}\iota_{b_{2}}\ldots \iota_{b_{n+1}}(\alpha_{n+1})
\end{equation}
and $b_{k+1} \neq b_k$ for all $k$.
Then $CS(\alpha)=b_1b_2\ldots b_k\ldots$
\end{algo}
We have the following properties which are straightforward to verify:

\begin{prop}
Suppose  $k_1, k_2, k_3 \in \QQ^{+}$, $\alpha \in \RR\setminus \{0,1\}$ and  $CS(\alpha)$ is defined as above.

\begin{enumerate}
    \item $\alpha \in \QQ$ is a cusp if and only if the above algorithm terminates in a finite number of steps, equivalently $CS(\alpha)$ is finite.
    \item If $\alpha$ is a rational or quadratic irrational,  then $\alpha$ is a fixed point of a hyperbolic element of $\Gamma$ if and only if the cutting sequence is infinite and eventually periodic. Equivalently, $\alpha_{k+n}=\alpha_k$ for some $k,n \in \NN$.
    \item Let $w_n=\iota_{b_1}\ldots \iota_{b_n}$ and $x_n=w_n(\infty)$. Then $\lim_{n \rightarrow \infty} x_n= \alpha$ if $\alpha$ is not a cusp, where $x_n \in \QQ$.
\end{enumerate}

\end{prop}

\noindent {\bf Examples:} 

\noindent For $\Gamma(1,1/5,5)$:
$CS(\frac{-1+\sqrt{5}}{2}) = \overline{3 1 2 3 1 3 2 1 3 1 }$, 
whereas for $CS(\sqrt{2})$ and $CS(\sqrt{3})$, no period was detected  up to length 10000.

\medskip
\noindent For $\Gamma(5/3,3/2,2/5)$: $CS(\frac{-1+\sqrt{5}}{2})$ is periodic with period 216 and $CS(\sqrt{2})$ is periodic with period 16, whereas for $CS(\sqrt{3})$, no period was detected up to length 10000. Details are given in Appendix B.

\medskip

\begin{algo}\label{alg:naivetransl} (Naive plus translation)
Here we incorporate the powers of the  translation $T_L:=\iota_3\iota_2\iota_1=\pm \left(
			\begin{array}{cc}
			1 & L \\
			0 & 1 \\
			\end{array}
			\right) $ into the algorithm where $L=1+k_3+1/k_1$. Of particular interest are the cases where $L \in \NN$, where the algorithm is closer to the standard continued fraction algorithm. We divide $\RR \setminus \{-1,0\}$ into the intervals $I_1'$, $I_2$, $I_3'$ and $O$ where 
			\begin{equation}\small{
			   I_1'=\left(\frac{-k_1-1}{k_1},-1\right),
			    \quad I_2'=(-1,0), \quad
			     I_3^{I}=(0, k_3], \quad O=\RR\setminus\left(\frac{-k_1-1}{k_1},k_3\right]. } 
			\end{equation}


			Define $\alpha_1=T_L^{-n_0}(\alpha_0)$ where $n_0\in \ZZ$ is chosen so that $\alpha_1 \in \left(\frac{-k_1-1}{k_1},k_3\right]$.
			
			The algorithm terminates if $\alpha_k \in \{-1,0, \infty\}$ for some $k\in \NN$, otherwise, define
			$$\alpha_{k+1}=T_L^{-n_k}\iota_{b_k}(\alpha_k)$$
			where $b_k$ is  the index of the interval which $\alpha_k$ lies in, and if $\iota_{b_k}(\alpha_k) \neq \infty$, then $n_k\in \ZZ$ is chosen so that $\alpha_{k+1} \in \left(\frac{-k_1-1}{k_1},k_3\right]$ and $n_k=0$ otherwise. Note that if $n_k=0$ and $b_{k+1}$ is defined, then $b_{k+1}\neq b_k$.
			Then, letting $B_j:=\iota_{b_j}T_L^{n_j}$, \begin{equation}\alpha=\alpha_0=T_L^{n_0}\iota_{b_1}T_L^{n_1}\ldots \iota_{b_k}T_L^{n_k}(\alpha_{k+1})=T_L^{n_0}B_1\ldots B_k(\alpha_{k+1}). \end{equation}
			Following standard notation for continued fractions, we call $\alpha_k$ the $(k+1)$-th {\it complete quotient} of $\alpha$. Here $w_k:=T_L^{n_0}\iota_{b_1}T_L^{n_1}\ldots \iota_{b_k}T_L^{n_k}$ is a word in $T_L, \iota_1, \iota_2, \iota_3$, and as a word in $\iota_1, \iota_2, \iota_3$ is not necessarily reduced. However, the cancellations that can occur are controlled, for example, the occurrences of $\iota_2$ do not cancel in the expression, and the reduced word length can be shown to be at least $k$. After reduction, we obtain the cutting sequence for $\alpha$.
\end{algo}

\medskip

\noindent {\bf Example}. For $\Gamma(5/3,3/2,2/5)$, taking $\alpha=\sqrt{2}$, we have
$$\alpha_0=\sqrt{2}= T_L\iota_2T_L\iota_1T_L^3\iota_1\iota_3\iota_1(\sqrt{2})$$ giving the periodic word $\overline{T_L\iota_2T_L\iota_1T_L^3\iota_1\iota_3\iota_1}$. 
Replacing $T_L$ by $\iota_3\iota_2\iota_1$, and reducing gives the same periodic sequence as that obtained by Algorithm \ref{alg:naive}. See Appendix B for the more interesting example where $\alpha=\frac{-1+\sqrt{5}}{2}$.


\medskip

\begin{algo}\label{alg:GCFWeierstrass}(Generalized continued fractions for integral Weierstrass groups)
We can modify Algorithm \ref{alg:naivetransl} slightly to obtain generalized continued fractions for $\alpha \in \RR$ associated to the integral  Weierstrass groups $\Gamma(1,1/n,n)$ with some nice properties. To start with, we use the fundamental domain $\triangledown_0:=(\infty, 0, n)=\iota_3(\triangle_0)$ with sides $s_1'=[\infty, 0]$, $s_2'=[0, n]$ and $s_3'=[n, \infty]$, colored by $1,2$ and $3$ respectively. Let $y_1,y_2,y_3$ be the marked points on the sides of $\triangledown_0$ and $\kappa_1, \kappa_2, \kappa_3$  the $\pi$-rotations about them respectively where 
\begin{equation}\small{
    \kappa_1= \frac{1}{\sqrt{n}}\left(
			\begin{array}{cc}
			0 & -n \\
			1 & 0 \\
			\end{array}
			\right) , \quad \kappa_2=\frac{1}{n} \left(
			\begin{array}{cc}
			n & -n^2 \\
			2 & -n \\
			\end{array}
			\right), \quad \kappa_3=\frac{1}{\sqrt{n}}\left(
			\begin{array}{cc}
			n & -n^2-n \\
			1 & -n \\
			\end{array}
			\right)}
\end{equation}
Let $W:=\kappa_3\kappa_2\kappa_1= \left(
			\begin{array}{cc}
			1 & L \\
			0 & 1 \\
			\end{array}
			\right)$ where $L=n+2$ and 
			$J_1'$, $J_2'$, $J_3'$ and $O$ be the intervals
			\begin{equation}\small{
			     J_1'=(-1,0), \quad J_2'=(0,n), \quad J_3'=(n, n+1], \quad O=\RR\setminus (-1,n+1]}
			\end{equation}

Then, using this modification, we can write
$\alpha=n_0L+\alpha_1$, and for $k \ge 1$,
\begin{equation}
    \alpha_k= 
    \begin{cases}
    \dfrac{-n}{n_kL+\alpha_{k+1}} & \text{if $\alpha_k \in J_1'$,}\\
    ~~\\
    \dfrac{n}{2}+\dfrac{-n^2/4}{-n/2+n_{k}L+\alpha_{k+1}} & \text{if $\alpha_k \in J_2'\setminus\{n/2\}$}\\
    ~~\\
    n+\dfrac{-n}{-n+n_{k}L+\alpha_{k+1}}  & \text{if $\alpha_k \in J_3'$}\\
    ~~\\
    \alpha_k & \text{if $\alpha_k \in \{0,n/2,n\}$}\\
    \end{cases} \\
\end{equation}
where $n_{k}$, $k \ge 0$ is chosen so that $\alpha_{k+1} \in (-1,n+1]$.
\end{algo}
This gives a generalized continued fraction for $\alpha \in \RR$ of the form
$$\alpha =b_0+\dfrac{a_1}{b_1+\dfrac{a_2}{b_2+\dfrac{a_3}{b_3+\ddots}}}.$$
We will also use the more compact notation
\begin{equation}\label{eqn:contfractionnotation}
    \alpha=b_0+\frac{a_1}{b_1+}~\frac{a_2}{b_2+}~\frac{a_3}{b_3+}~\cdots \quad \hbox{and} \quad  \alpha=b_0+\cdots \overline{\frac{a_k}{b_k+}~\frac{a_{k+1}}{b_{k+1}+}\cdots \frac{a_{k+n-1}}{b_{k+n-1}+}}
\end{equation}
where the second expression is eventually periodic with period $n$. The partial numerators $a_k \in \{- n, - n^2/4\}$ and the partial denominators $b_k  \in \ZZ/2$.
In particular, if $n$ is even, $b_0 \in \ZZ$ and  $a_k, b_k \in \ZZ\setminus\{0\}$ for all $k \in \NN$. The continued fraction terminates if and only if $\alpha$ is a cusp, and if not, the successive convergents converge to $\alpha$.

\medskip

\noindent {\bf Example:}  For $\Gamma=\Gamma(1,1/3,3)$ and $\alpha=1$, we have $$1=\alpha=\alpha_0=\alpha_1=\kappa_2 W(\alpha_2)=\kappa_2 W(1)=\kappa_2W\kappa_2W\kappa_2W\cdots$$
which gives the continued fraction (after contracting terms to obtain the $b_i$'s)
$$1=\frac{3}{2}+\frac{-9/4}{5~+}~\frac{-9/4}{5~+}~\frac{-9/4}{5~+}~\cdots=\frac{3}{2}+\overline{\frac{-9/4}{5~+}}$$

 For $\Gamma=\Gamma(1,1/5,5)$ and $\alpha=1$, similarly, we have
$$1=\frac{5}{2}+\frac{-25/4}{9/2~+}~\frac{-5}{33/2~+}~\frac{-25/4}{9/2~+}~\frac{-5}{33/2~+}~\cdots=\frac{5}{2}+\overline{\frac{-25/4}{9/2~+}~\frac{-5}{33/2~+}}$$



\medskip
\subsection{Algorithms for Pseudomodular jigsaw groups}\label{ss:algoPSG}
Let $J$ be an integral jigsaw, $\Gamma_J$ the jigsaw group where $|J|=N$. There is a coloring of the sides of $J$ (in cyclic order) by $\{1,2, \ldots, N+2\}$ inducing a coloring of the edges of $\QQQ_J$. Hence for any $\alpha \in \RR \setminus V(J)$ where $V(J)$ is the vertex set of $J$, we have an induced cutting sequence $CS(\alpha)$ as before.   
\begin{algo}(Naive algorithm)\label{algo:naivepseudo} Essentially the same algorithm as algorithm \ref{alg:naive} works where the coloring set $\{1,2,3\}$ is replaced by $\{1,2, \ldots, N+2\}$ and the set of intervals $\{I_1,I_2,I_3\}$ by $\{I_1, I_2, \ldots, I_{N+2}\}$ with corresponding involutions $\iota_1, \ldots, \iota_{N+1}$ to produce the cutting sequence for $\alpha$ with respect to $\Gamma_J$.
\end{algo}
As before, $\alpha$ is a fixed point of a hyperbolic element of $\Gamma_J$ if and only if the cutting sequence is eventually periodic.

\medskip
\noindent {\bf Examples:} For the diamond $\SSS(1,2)$ jigsaw, we have the following cutting sequences for the following quadratic irrationals:
$$CS(\frac{-1+\sqrt{5}}{2}) = \overline{321231}, \qquad
CS(\sqrt{5}) = \overline{424121}.$$
$CS(\sqrt{7})$ is periodic with period $84$.
For $CS(\sqrt{2})$ and $CS(\sqrt{3})$, no period was detected up to length 10000.
Details are in Appendix B.


\medskip

We next describe a pseudo-euclidean algorithm for the $\SSS(1,2)$ diamond jigsaw group and the associated generalized continued fraction algorithm, with some rather nice properties, using the killer intervals. The algorithm is not canonical, some choices are involved, but once these are made once and for all, the algorithm is completely determined.

\begin{algo}\label{alg:pseudoalgdiamond}(Pseudo-euclidean algorithm).
Let $J$ be the diamond $\SSS(1,2)$ jigsaw in normalized position with sides $s_1=[-\infty, -1]$, $s_2=[-1,0]$, $s_3=[0,1]$ and $s_4=[1, \infty]$ and $\iota_1,\iota_2, \iota_3, \iota_4$ the $\pi$-rotations about the marked points $x_1,x_2,x_3,x_4$ respectively. Let $$V:=\iota_4\iota_3\iota_2\iota_1=\pm \left(
			\begin{array}{cc}
			1 & 7 \\
			0 & 1 \\
			\end{array}
			\right)$$ and use as the fundamental interval $I:=(-3,4]$   divided into the following  intervals $A_k$, $k=1,2,3,4$ where the interior of $A_k$ is contained in the Killer interval of  the involution $\epsilon_k\in \Gamma_J$ about the marked point $p_k$:
\begin{eqnarray}
A_1=(-3,-2], \qquad \epsilon_1 &=\iota_1\iota_2\iota_1= \left(
			\begin{array}{cc}
			-2 & -5 \\
			1 & 2 \\
			\end{array}
			\right), \quad p_1=-2+i, \\
A_2= (-2,0], \qquad \epsilon_2 &= \iota_1=\left(
			\begin{array}{cc}
			-1 & -2 \\
			1 & 1 \\
			\end{array}
			\right), \quad p_2=-1+i, \\
A_3=(0,2], \qquad \epsilon_3 &=\iota_4= \left(
			\begin{array}{cc}
			1 & -3 \\
			1 & -1 \\
			\end{array}
			\right), \quad p_3=1+2i, \\
A_4=(2,4], \qquad \epsilon_4 &=\iota_4\iota_3\iota_4=  \left(
			\begin{array}{cc}
			3 & -11 \\
			1 & -3 \\
			\end{array}
			\right), \quad p_4=3+2i.
\end{eqnarray}

Similar to Algorithm \ref{alg:naivetransl}, given $\alpha=\alpha_0 \in \RR$, define $\alpha_1=V^{-n_0}(\alpha_0)$ where $n_0\in \ZZ$ is chosen so that $\alpha_1 \in I$.
			
			The algorithm terminates if $\alpha_k =\infty$, otherwise, define
			$$\alpha_{k+1}=V^{-n_k}\epsilon_{b_k}(\alpha_k)$$
			where $b_k$ is  the index of the sub-interval which $\alpha_k$ lies in, and if $\epsilon_{b_k}(\alpha_k) \neq \infty$, then $n_k\in \ZZ$ is chosen so that $\alpha_{k+1} \in I$ and $n_k=0$ otherwise.
			Then, letting $B_j:=\epsilon_{b_j}V^{n_j}$, \begin{equation}\label{eqn:VB}\alpha=\alpha_0=V^{n_0}\epsilon_{b_1}V^{n_1}\ldots \epsilon_{b_k}V^{n_k}(\alpha_{k+1})=V^{n_0}B_1\ldots B_k(\alpha_{k+1}). \end{equation}

\end{algo}

This algorithm has the property that if $\alpha\in \QQ$, then the complete quotients $\alpha_k$ of $\alpha$ have   denominators $q_k$ which are decreasing, strict up to when the denominator is one, after which it terminates at $0$ after at most $2$ steps ($\alpha_{k+1}$ or $\alpha_{k+2}=\infty$). On the other hand, if $\alpha \notin \QQ$, then  $x_n=V^{n_0}B_1\ldots B_n(\infty)$ converge to $\alpha$. We call $\{x_n\}$ the convergents of $\alpha$ with respect to this pseudo-euclidean algorithm. This pseudo-euclidean algorithm is equivalent to the following generalized continued fraction algorithm:

\begin{algo}\label{alg:contfractdiamond}(Continued fraction from the diamond $\SSS(1,2)$ jigsaw group).
Let $\alpha \in \RR$ and choose $n_0\in \ZZ$ such that
$\alpha=7n_0+\alpha_1$, where $\alpha_1 \in I=(-3,4]$, and for $k \ge 1$,
\begin{equation}\label{eqn:CFPMG}
    \alpha_k= 
    \begin{cases}
    -2+\dfrac{-1}{2+7n_k+\alpha_{k+1}} & \text{if $\alpha_k \in A_1$,}\\
    ~~\\
    ~~-1+\dfrac{-1}{1+7n_k+\alpha_{k+1}} & \text{if $\alpha_k \in A_2$,}\\
    ~~\\
    1+\dfrac{-2}{-1+7n_k+\alpha_{k+1}} & \text{if $\alpha_k \in A_3$,}\\
    ~~\\
    3+\dfrac{-2}{-3+7n_k+\alpha_{k+1}} & \text{if $\alpha_k \in A_4$,}\\
    \end{cases} \\
\end{equation}
where $n_k=0$ if $\alpha_{k+1}=\infty$, and $n_{k}$ is chosen so that $\alpha_{k+1} \in I$ otherwise.
\end{algo}
This gives a generalized continued fraction for $\alpha \in \RR$ of the form, using the notation in (\ref{eqn:contfractionnotation})
$$\alpha= b_0+\frac{a_1}{b_1+}~\frac{a_2}{b_2+}~\frac{a_3}{b_3+}~\cdots$$

where the partial numerators $a_k \in \{- 1, - 2\}$ and the partial denominators $b_k  \in \ZZ \setminus\{0\}$ for $k \ge 1$. 
The continued fraction terminates if and only if $\alpha\in \QQ$ , and if not, the successive convergents converge to $\alpha$.
\medskip

A more direct comparison with the usual continued fraction expansion can be seen by expressing (\ref{eqn:VB}) in terms of
\begin{equation}\label{eqn:defnTSR}
    T:=\left(
			\begin{array}{cc}
			1 & 1 \\
			0 & 1 \\
			\end{array}
			\right) , \quad S:=\left(
			\begin{array}{cc}
			0 & -1 \\
			1 & 0 \\
			\end{array}
			\right) , \quad \hbox{and}\quad R:=\left(
			\begin{array}{cc}
			0 & -2 \\
			1 & 0 \\
			\end{array}
			\right).
\end{equation}
Suppose $\alpha \in \RR \setminus \QQ$. Note that $$V=T^7, \quad \epsilon_1=T^{-2}ST^2,\quad  \epsilon_2=T^{-1}ST, \quad \epsilon_3=TRT^{-1}, \quad \epsilon_4=T^3RT^{-3},$$ so (\ref{eqn:VB}) becomes
\begin{equation}\label{eqn:CFGroup}
    \alpha=\alpha_0=T^{b_0}U_1T^{b_1}U_2\ldots U_{k}T^{b_k}(\alpha_{k+1})
\end{equation}
where $b_i \in \ZZ\setminus \{0\}$, $U_i \in \{S,R\}$, and $\alpha_{k+1} \in (-3,4]$.

Then the generalized continued fraction for $\alpha$ using this algorithm has partial numerator $a_k=-1$ or $-2$ depending on whether $U_k=S$ or $R$ respectively, and partial denominators $b_k$ as given in expression (\ref{eqn:CFGroup}).

\begin{appendices}
\section{Test of arithmeticity and arithmetic $\SSS(m,n)$ diamond jigsaws}
We list here the traces of the relevant elements necessary to verify that a particular jigsaw group is arithmetic. 

For $n=3$ or $9$, recall from the proof of Proposition \ref{prop:3,5,9,25}  that the star jigsaw group $\Gamma_J=\langle \iota_0, \iota_1, \ldots, \iota_5 \rangle$, with subgroups
$$H_J:=\langle \gamma_1, \dots, \gamma_{5}\rangle, \quad  H_J^{(SQ)}:=\langle \gamma_1^2, \ldots, \gamma_5^2 \rangle \quad \hbox{where} \quad \gamma_k=\iota_0\iota_k,~ k=1, \ldots, 5. $$
We need to check that $\tr \gamma \in \ZZ$ for $\gamma$ of the form (case 1)
$$\gamma=\gamma_{\nu_1}^2\gamma_{\nu_2}^2\ldots \gamma_{\nu_k}^2, \quad \hbox{where} \quad \nu_1<\nu_2 \cdots <\nu_k; \quad k \le 5$$
 and  $(\tr \gamma)^2 \in \ZZ$ for $\gamma$ of the form (case 2)
 $$\gamma=\gamma_{\nu_1}\gamma_{\nu_2}\ldots \gamma_{\nu_k}, \quad \hbox{where} \quad \nu_1<\nu_2 \cdots <\nu_k; \quad k \le 5.$$

For $n=5$ or $25$, we do the same for the diamond jigsaw group, here $H_J$ is generated by $3$ elements and we only need to check the above for the $7$ words of length at most 3. Tables \ref{table:n5},\ref{table:n25},\ref{table:n3},\ref{table:n9}  give the list of the elements and their traces. Note that the elements are shown as elements of $\PGLtwoQ$, we need to divide by the square root of the determinant for the trace calculation.



More generally, we may also consider the diamond $\SSS(m,n)$ jigsaws where $1 \le m <n \le N_0$ and run the above test for arithmeticity. Omitting details, for $N_0=10000$, we have:

\begin{prop}\label{prop:mn}
    {Table \ref{tab:diamonds} gives a list of the pairs $(m,n)$ where $1 \le m \le n \le 10000$ for which the diamond $\SSS(m,n)$ jigsaw is arithmetic.}
    
    
    
    
    
    
    
    
    
    
    
    
    
    
    
    
    
    
    
    
    
    
    
    
    
    
    
    
    
    
    
    
    
    
    
    
    
    
    \begin{table}[h!]
    \fontsize{9}{11}\selectfont
        \centering
       \hspace*{-0.2cm} \begin{tabular}{|c|l|}
        \hline
        $l$ & (m,n) pairs\\
        \hline
          5  &  (16, 20) (20, 36) (36, 80) (80, 196) (196, 500) (500, 1296) (1296, 3380) (3380, 8836) \\
        \hline
        6  &  (8, 12) (12, 32) (32, 108) (108, 392) (392, 1452) (1452, 5408) \\
        \hline
        8  &  (4, 8) (8, 36) (36, 200) (200, 1156) (1156, 6728) \\
        \hline
        9  &  (4, 4) (4, 16) (16, 100) (100, 676) (676, 4624) \\
        \hline
        12  &  (2, 6) (6, 50) (50, 486) (486, 4802) \\
        \hline
        16  &  (2, 2) (2, 18) (18, 242) (242, 3362) \\
        \hline
        20  &  (1, 5) (5, 81) (81, 1445) \\
        \hline
        36  &  (1, 1) (1, 25) (25, 841) \\
        \hline
        \end{tabular}
        
        \medskip
        
        \caption{Pairs $(m,n)\in \NN \times \NN$, $1 \le m \le n \le 10,000$ for which $\SSS(m,n)$ is arithmetic, grouped by $l:=(m+n+4)^2/mn$}
        \label{tab:diamonds}
    \end{table}
\end{prop}

\begin{table}[!ht]
\centering
\begin{tabular}{ |c|c|c|c|c| } 
\hline
Word & Element (case 1) & Trace & Element (case 2) & Trace \\
 \hline
$\gamma_1^2$, $(\gamma_1)^2$  &  (11, 15, -3, -4)  &  7  &  (11, 15, -3, -4)  &  7 \\
\hline
$\gamma_2^2$, $(\gamma_2)^2$  &  (59, 45, -21, -16)  &  43  &  (59, 45, -21, -16)  &  43 \\
\hline
$\gamma_1^2\gamma_2^2$, $(\gamma_1\gamma_2)^2$  &  (314, 245, -91, -71)  &  243  &  (334, 255, -93, -71)  &  263 \\
\hline
$\gamma_3^2$, $(\gamma_3)^2$  &  (9, -20, -4, 9)  &  18  &  (9, -20, -4, 9)  &  18 \\
\hline
$\gamma_1^2\gamma_3^2$, $(\gamma_1\gamma_3)^2$  &  (29, -45, -9, 14)  &  43  &  (39, -85, -11, 24)  &  63 \\
\hline
$\gamma_2^2\gamma_3^2$, $(\gamma_2\gamma_3)^2$  &  (321, -575, -115, 206)  &  527  &  (351, -775, -125, 276)  &  627 \\
\hline
$\gamma_1^2\gamma_2^2\gamma_3^2$, $(\gamma_1\gamma_2\gamma_3)^2$  &  (1456, -2585, -423, 751)  &  2207  &  (1986, -4385, -553, 1221)  &  3207 \\
\hline
\end{tabular}

\medskip

\caption{Arithmeticity check for `diamond' jigsaws with $n=5$}
\label{table:n5}
\end{table}

\begin{table}[!ht]
\fontsize{9}{11}\selectfont
\centering
\hspace*{-0.1cm}
\begin{tabular}{ |c|c|c|c|c| } 
\hline
Word & Element (case 1) & Trace & Element (case 2) & Trace \\
 \hline
$\gamma_1^2$, $(\gamma_1)^2$  &  (11, 15, -3, -4)  &  7  &  (11, 15, -3, -4)  &  7 \\
\hline
$\gamma_2^2$, $(\gamma_2)^2$  &  (875, 825, -297, -280)  &  119  &  (875, 825, -297, -280)  &  119 \\
\hline
$\gamma_1^2\gamma_2^2$, $(\gamma_1\gamma_2)^2$  &  (4990, 4725, -1431, -1355)  &  727  &  (5170, 4875, -1437, -1355)  &  763 \\
\hline
$\gamma_3^2$, $(\gamma_3)^2$  &  (25, -300, -12, 145)  &  34  &  (25, -300, -12, 145)  &  34 \\
\hline
$\gamma_1^2\gamma_3^2$, $(\gamma_1\gamma_3)^2$  &  (65, -525, -21, 170)  &  47  &  (95, -1125, -27, 320)  &  83 \\
\hline
$\gamma_2^2\gamma_3^2$, $(\gamma_2\gamma_3)^2$  &  (1945, -16575, -663, 5650)  &  1519  &  (2395, -28575, -813, 9700)  &  2419 \\
\hline
$\gamma_1^2\gamma_2^2\gamma_3^2$, $(\gamma_1\gamma_2\gamma_3)^2$  &  (9080, -77025, -2607, 22115)  &  6239  &  (14150, -168825, -3933, 46925)  &  12215 \\
\hline
\end{tabular}

\medskip

\caption{Arithmeticity check for `diamond' jigsaws with $n=25$}
\label{table:n25}
\end{table}

\begin{table}[!ht]
\fontsize{6}{11}\selectfont
\hspace*{-2.8cm}
\begin{tabular}{ |c|c|c|c|c| } 
\hline
Word & Element (case 1) & Trace & Element (case 2) & Trace \\
\hline
$\gamma_1^2$, $(\gamma_1)^2$  &  (24, -5, 5, -1)  &  23  &  (24, -5, 5, -1)  &  23 \\
\hline
$\gamma_2^2$, $(\gamma_2)^2$  &  (109, 33, 33, 10)  &  119  &  (109, 33, 33, 10)  &  119 \\
\hline
$\gamma_1^2\gamma_2^2$, $(\gamma_1\gamma_2)^2$  &  (2349, 700, 500, 149)  &  2498  &  (2451, 742, 512, 155)  &  2606 \\
\hline
$\gamma_3^2$, $(\gamma_3)^2$  &  (153, 112, 56, 41)  &  194  &  (153, 112, 56, 41)  &  194 \\
\hline
$\gamma_1^2\gamma_3^2$, $(\gamma_1\gamma_3)^2$  &  (3008, 2183, 649, 471)  &  3479  &  (3392, 2483, 709, 519)  &  3911 \\
\hline
$\gamma_2^2\gamma_3^2$, $(\gamma_2\gamma_3)^2$  &  (18177, 13261, 5513, 4022)  &  22199  &  (18525, 13561, 5609, 4106)  &  22631 \\
\hline
$\gamma_1^2\gamma_2^2\gamma_3^2$, $(\gamma_1\gamma_2\gamma_3)^2$  &  (379325, 276716, 80764, 58917)  &  438242  &  (416555, 304934, 87016, 63699)  &  480254 \\
\hline
$\gamma_4^2$, $(\gamma_4)^2$  &  (76, 99, 33, 43)  &  119  &  (76, 99, 33, 43)  &  119 \\
\hline
$\gamma_1^2\gamma_4^2$, $(\gamma_1\gamma_4)^2$  &  (1311, 1681, 287, 368)  &  1679  &  (1659, 2161, 347, 452)  &  2111 \\
\hline
$\gamma_2^2\gamma_4^2$, $(\gamma_2\gamma_4)^2$  &  (8689, 11220, 2640, 3409)  &  12098  &  (9373, 12210, 2838, 3697)  &  13070 \\
\hline
$\gamma_1^2\gamma_2^2\gamma_4^2$, $(\gamma_1\gamma_2\gamma_4)^2$  &  (175482, 226567, 37367, 48245)  &  223727  &  (210762, 274555, 44027, 57353)  &  268115 \\
\hline
$\gamma_3^2\gamma_4^2$, $(\gamma_3\gamma_4)^2$  &  (15048, 19519, 5513, 7151)  &  22199  &  (15324, 19963, 5609, 7307)  &  22631 \\
\hline
$\gamma_1^2\gamma_3^2\gamma_4^2$, $(\gamma_1\gamma_3\gamma_4)^2$  &  (280331, 363593, 60499, 78468)  &  358799  &  (339731, 442577, 71011, 92508)  &  432239 \\
\hline
$\gamma_2^2\gamma_3^2\gamma_4^2$, $(\gamma_2\gamma_3\gamma_4)^2$  &  (1751001, 2271148, 531080, 688841)  &  2439842  &  (1855413, 2417098, 561782, 731849)  &  2587262 \\
\hline
$\gamma_1^2\gamma_2^2\gamma_3^2\gamma_4^2$, $(\gamma_1\gamma_2\gamma_3\gamma_4)^2$  &  (35375534, 45884051, 7531999, 9769425)  &  45144959  &  (41721002, 54351107, 8715283, 11353641)  &  53074643 \\
\hline
$\gamma_5^2$, $(\gamma_5)^2$  &  (9, 25, 5, 14)  &  23  &  (9, 25, 5, 14)  &  23 \\
\hline
$\gamma_1^2\gamma_5^2$, $(\gamma_1\gamma_5)^2$  &  (125, 308, 28, 69)  &  194  &  (191, 530, 40, 111)  &  302 \\
\hline
$\gamma_2^2\gamma_5^2$, $(\gamma_2\gamma_5)^2$  &  (942, 2419, 287, 737)  &  1679  &  (1146, 3187, 347, 965)  &  2111 \\
\hline
$\gamma_1^2\gamma_2^2\gamma_5^2$, $(\gamma_1\gamma_2\gamma_5)^2$  &  (18035, 46259, 3841, 9852)  &  27887  &  (25769, 71663, 5383, 14970)  &  40739 \\
\hline
$\gamma_3^2\gamma_5^2$, $(\gamma_3\gamma_5)^2$  &  (1769, 4661, 649, 1710)  &  3479  &  (1937, 5393, 709, 1974)  &  3911 \\
\hline
$\gamma_1^2\gamma_3^2\gamma_5^2$, $(\gamma_1\gamma_3\gamma_5)^2$  &  (30301, 79788, 6540, 17221)  &  47522  &  (42943, 119562, 8976, 24991)  &  67934 \\
\hline
$\gamma_2^2\gamma_3^2\gamma_5^2$, $(\gamma_2\gamma_3\gamma_5)^2$  &  (199466, 525323, 60499, 159333)  &  358799  &  (234530, 652979, 71011, 197709)  &  432239 \\
\hline
$\gamma_1^2\gamma_2^2\gamma_3^2\gamma_5^2$, $(\gamma_1\gamma_2\gamma_3\gamma_5)^2$  &  (3817723, 10054479, 812853, 2140756)  &  5958479  &  (5273665, 14682951, 1101639, 3067186)  &  8340851 \\
\hline
$\gamma_4^2\gamma_5^2$, $(\gamma_4\gamma_5)^2$  &  (1149, 3100, 500, 1349)  &  2498  &  (1179, 3286, 512, 1427)  &  2606 \\
\hline
$\gamma_1^2\gamma_4^2\gamma_5^2$, $(\gamma_1\gamma_4\gamma_5)^2$  &  (17534, 47261, 3841, 10353)  &  27887  &  (25736, 71729, 5383, 15003)  &  40739 \\
\hline
$\gamma_2^2\gamma_4^2\gamma_5^2$, $(\gamma_2\gamma_4\gamma_5)^2$  &  (122979, 331573, 37367, 100748)  &  223727  &  (145407, 405265, 44027, 122708)  &  268115 \\
\hline
$\gamma_1^2\gamma_2^2\gamma_4^2\gamma_5^2$, $(\gamma_1\gamma_2\gamma_4\gamma_5)^2$  &  (2346761, 6327224, 499720, 1347321)  &  3694082  &  (3269633, 9112820, 683008, 1903617)  &  5173250 \\
\hline
$\gamma_3^2\gamma_4^2\gamma_5^2$, $(\gamma_3\gamma_4\gamma_5)^2$  &  (220445, 594476, 80764, 217797)  &  438242  &  (237731, 662582, 87016, 242523)  &  480254 \\
\hline
$\gamma_1^2\gamma_3^2\gamma_4^2\gamma_5^2$, $(\gamma_1\gamma_3\gamma_4\gamma_5)^2$  &  (3766462, 10157001, 812853, 2192017)  &  5958479  &  (5270464, 14689353, 1101639, 3070387)  &  8340851 \\
\hline
$\gamma_2^2\gamma_3^2\gamma_4^2\gamma_5^2$, $(\gamma_2\gamma_3\gamma_4\gamma_5)^2$  &  (24833423, 66968273, 7531999, 20311536)  &  45144959  &  (28784207, 80224697, 8715283, 24290436)  &  53074643 \\
\hline
$\gamma_1^2\gamma_2^2\gamma_3^2\gamma_4^2\gamma_5^2$, $(\gamma_1\gamma_2\gamma_3\gamma_4\gamma_5)^2$  &  (474352033, 1279184712, 100996896, 272358241)  &  746710274  &  (647244553, 1803940548, 135205752, 376833049)  &  1024077602 \\
\hline
\end{tabular}
\medskip

\caption{Arithmeticity check for `star' jigsaws with $n=3$.}
\label{table:n3}
\end{table}

\begin{table}[!ht]
\fontsize{4}{11}\selectfont
\centering
\hspace*{-2.8cm}
\begin{tabular}{ |c|c|c|c|c| } 
\hline
Word & Element (case 1) & Trace & Element (case 2) & Trace \\
 \hline
$\gamma_1^2$, $(\gamma_1)^2$  &  (120, -11, 11, -1)  &  119  &  (120, -11, 11, -1)  &  119 \\
\hline
$\gamma_2^2$, $(\gamma_2)^2$  &  (219, 35, 25, 4)  &  223  &  (219, 35, 25, 4)  &  223 \\
\hline
$\gamma_1^2\gamma_2^2$, $(\gamma_1\gamma_2)^2$  &  (25865, 4104, 2376, 377)  &  26242  &  (26005, 4156, 2384, 381)  &  26386 \\
\hline
$\gamma_3^2$, $(\gamma_3)^2$  &  (281, 240, 48, 41)  &  322  &  (281, 240, 48, 41)  &  322 \\
\hline
$\gamma_1^2\gamma_3^2$, $(\gamma_1\gamma_3)^2$  &  (31048, 26413, 2867, 2439)  &  33487  &  (33192, 28349, 3043, 2599)  &  35791 \\
\hline
$\gamma_2^2\gamma_3^2$, $(\gamma_2\gamma_3)^2$  &  (61123, 52059, 6993, 5956)  &  67079  &  (63219, 53995, 7217, 6164)  &  69383 \\
\hline
$\gamma_1^2\gamma_2^2\gamma_3^2$, $(\gamma_1\gamma_2\gamma_3)^2$  &  (7116129, 6060760, 653720, 556769)  &  7672898  &  (7506893, 6411596, 688192, 587781)  &  8094674 \\
\hline
$\gamma_4^2$, $(\gamma_4)^2$  &  (184, 205, 35, 39)  &  223  &  (184, 205, 35, 39)  &  223 \\
\hline
$\gamma_1^2\gamma_4^2$, $(\gamma_1\gamma_4)^2$  &  (19599, 21707, 1813, 2008)  &  21607  &  (21695, 24171, 1989, 2216)  &  23911 \\
\hline
$\gamma_2^2\gamma_4^2$, $(\gamma_2\gamma_4)^2$  &  (39129, 43400, 4480, 4969)  &  44098  &  (41521, 46260, 4740, 5281)  &  46802 \\
\hline
$\gamma_1^2\gamma_2^2\gamma_4^2$, $(\gamma_1\gamma_2\gamma_4)^2$  &  (4533324, 5028005, 416455, 461899)  &  4995223  &  (4930380, 5493109, 451991, 503579)  &  5433959 \\
\hline
$\gamma_3^2\gamma_4^2$, $(\gamma_3\gamma_4)^2$  &  (59896, 66661, 10235, 11391)  &  71287  &  (60104, 66965, 10267, 11439)  &  71543 \\
\hline
$\gamma_1^2\gamma_3^2\gamma_4^2$, $(\gamma_1\gamma_3\gamma_4)^2$  &  (6508183, 7243123, 600989, 668856)  &  7177039  &  (7099543, 7909971, 650877, 725176)  &  7824719 \\
\hline
$\gamma_2^2\gamma_3^2\gamma_4^2$, $(\gamma_2\gamma_3\gamma_4)^2$  &  (12876529, 14330680, 1473200, 1639569)  &  14516098  &  (13522121, 15065700, 1543668, 1719881)  &  15242002 \\
\hline
$\gamma_1^2\gamma_2^2\gamma_3^2\gamma_4^2$, $(\gamma_1\gamma_2\gamma_3\gamma_4)^2$  &  (1491857772, 1660333661, 137048863, 152525827)  &  1644383599  &  (1605674172, 1788965309, 147199663, 164002819)  &  1769676991 \\
\hline
$\gamma_5^2$, $(\gamma_5)^2$  &  (3, 11, 1, 4)  &  7  &  (3, 11, 1, 4)  &  7 \\
\hline
$\gamma_1^2\gamma_5^2$, $(\gamma_1\gamma_5)^2$  &  (257, 696, 24, 65)  &  322  &  (349, 1276, 32, 117)  &  466 \\
\hline
$\gamma_2^2\gamma_5^2$, $(\gamma_2\gamma_5)^2$  &  (548, 1557, 63, 179)  &  727  &  (692, 2549, 79, 291)  &  983 \\
\hline
$\gamma_1^2\gamma_2^2\gamma_5^2$, $(\gamma_1\gamma_2\gamma_5)^2$  &  (61751, 175119, 5673, 16088)  &  77839  &  (82171, 302679, 7533, 27748)  &  109919 \\
\hline
$\gamma_3^2\gamma_5^2$, $(\gamma_3\gamma_5)^2$  &  (987, 3107, 169, 532)  &  1519  &  (1083, 4051, 185, 692)  &  1775 \\
\hline
$\gamma_1^2\gamma_3^2\gamma_5^2$, $(\gamma_1\gamma_3\gamma_5)^2$  &  (97113, 305384, 8968, 28201)  &  125314  &  (127925, 478508, 11728, 43869)  &  171794 \\
\hline
$\gamma_2^2\gamma_3^2\gamma_5^2$, $(\gamma_2\gamma_3\gamma_5)^2$  &  (198084, 622973, 22663, 71275)  &  269359  &  (243652, 911389, 27815, 104043)  &  347695 \\
\hline
$\gamma_1^2\gamma_2^2\gamma_3^2\gamma_5^2$, $(\gamma_1\gamma_2\gamma_3\gamma_5)^2$  &  (22248239, 69970215, 2043825, 6427784)  &  28676023  &  (28932275, 108222207, 2652357, 9921236)  &  38853511 \\
\hline
$\gamma_4^2\gamma_5^2$, $(\gamma_4\gamma_5)^2$  &  (713, 2312, 136, 441)  &  1154  &  (757, 2844, 144, 541)  &  1298 \\
\hline
$\gamma_1^2\gamma_4^2\gamma_5^2$, $(\gamma_1\gamma_4\gamma_5)^2$  &  (66964, 216825, 6195, 20059)  &  87023  &  (89256, 335329, 8183, 30743)  &  119999 \\
\hline
$\gamma_2^2\gamma_4^2\gamma_5^2$, $(\gamma_2\gamma_4\gamma_5)^2$  &  (138039, 447035, 15805, 51184)  &  189223  &  (170823, 641771, 19501, 73264)  &  244087 \\
\hline
$\gamma_1^2\gamma_2^2\gamma_4^2\gamma_5^2$, $(\gamma_1\gamma_2\gamma_4\gamma_5)^2$  &  (15482321, 50138640, 1422288, 4606001)  &  20088322  &  (20284249, 76206616, 1859552, 6986217)  &  27270466 \\
\hline
$\gamma_3^2\gamma_4^2\gamma_5^2$, $(\gamma_3\gamma_4\gamma_5)^2$  &  (227105, 735768, 38808, 125729)  &  352834  &  (247277, 929004, 42240, 158693)  &  405970 \\
\hline
$\gamma_1^2\gamma_3^2\gamma_4^2\gamma_5^2$, $(\gamma_1\gamma_3\gamma_4\gamma_5)^2$  &  (22270852, 72152113, 2056571, 6662787)  &  28933639  &  (29208600, 109734857, 2677807, 10060351)  &  39268951 \\
\hline
$\gamma_2^2\gamma_3^2\gamma_4^2\gamma_5^2$, $(\gamma_2\gamma_3\gamma_4\gamma_5)^2$  &  (45475199, 147328595, 5202805, 16855824)  &  62331023  &  (55632063, 209006131, 6350885, 23859872)  &  79491935 \\
\hline
$\gamma_1^2\gamma_2^2\gamma_3^2\gamma_4^2\gamma_5^2$, $(\gamma_1\gamma_2\gamma_3\gamma_4\gamma_5)^2$  &  (5101478761, 16527551232, 468645120, 1518296281)  &  6619775042  &  (6605987825, 24818277128, 605601808, 2275207569)  &  8881195394 \\
\hline
\end{tabular}

\medskip

\caption{Arithmeticity check for `star' jigsaws with $n=9$.}
\label{table:n9}
\end{table}

\section{Cutting sequences }

Here we show some of the cutting sequences obtained by coding the algorithms in section \ref{s:PSAlgorithms}.

\medskip

\noindent {\bf Algorithm \ref{alg:naive}:} For the group $\Gamma(5/3,3/2,2/5)$,
$$CS(\sqrt{2})= \overline{\seqsplit{3212323213213231}}.$$
$CS(\frac{-1+\sqrt{5}}{2})$ is  periodic  with fundamental period

\noindent \seqsplit{ 323231321321321321321321321321321321321321321321321321321321321321321321321321321321321321321321321321321321321321321321321321321321321321321321321321321321321321321321321321321321321321232323212123123123123123123121}.

\noindent $CS (\sqrt{3})$ has cutting sequence starting with

\noindent \seqsplit{ 3121312131231213213123121323123123123123213213213213213213213213121312312312132132132132132132132132131312312312312312312312312312312312312312312312312312312312312312312312312312312312312312312312312312312312312312312312312312312312312312312312312312312312312312312312312312312312...}.
where no period was detected up to length 10000.

\medskip

\noindent {\bf Algorithm \ref{alg:naivetransl}:} For the group $\Gamma(5/3,3/2,2/5)$, 
$$\alpha=\alpha_0=\frac{-1+\sqrt{5}}{2} = \overline{T_L\iota_1 T_L\iota_1 \iota_3 \iota_1 T_L^{60}\iota_2 T_L\iota_1 T_L\iota_1 T_L\iota_2 T_L^{-7}\iota_3 \iota_1} \left(\frac{-1+\sqrt{5}}{2}\right)$$
which gives the periodic word of period 216 in $\iota_1,\iota_2,\iota_3$ obtained by Algorithm \ref{alg:naive} after expansion and reduction.

\medskip

\noindent {\bf  Algorithm \ref{algo:naivepseudo}:} For the diamond $\SSS(1,2)$ jigsaw with vertices $\{\infty, -1,0,1\}$:\\
$CS(\sqrt{7})$ is periodic with  fundamental period

\noindent \seqsplit{421412143214123123412343212342341431234341243234134121434123412412321432123213212121}.

\medskip

\noindent $CS(\sqrt{2})$ has  cutting sequence starting with

\noindent \seqsplit{ 412412341323421341234123412321421231234123412341234123412341212124321234123141341231234123412342343212434...}.

\noindent where no period was detected up to length 10000.
\medskip

\noindent $CS(\sqrt{3})$ has  cutting sequence starting with

\noindent \seqsplit{ 413123432313212141412341234143424214321342143214314214212341234123412341234123412341234121232121241232142...}.

\noindent where no period was detected up to length 10000.
\medskip

\end{appendices}


\begin{thebibliography}{99}

	\bibitem{Gar}
	C. Galaz-Garcia, \emph{New examples from the jigsaw group construction}, preprint (2020) Arxiv:2005.11601.
	
	\bibitem{HLM}
	H. M. Hilden, M.T. Lozano and J.M. Montesinos, {\emph A characterization of arithmetic subgroups of $\SLtwoR$
	and $\SLtwoC$}, Math. Nachr. {\bf 159}, (1992), 245--270.
	
	\bibitem{Hor} R. Horowitz, {\emph Characters of free groups represented in the   linear group}, Comm. Pure Appl. Math. {\bf 25}, (1972) 635-649.
	
	\bibitem{LR}
	D.D.  Long and A.W. Reid, \emph{Pseudomodular surfaces}, J. Reine Angew. Math. {\bf 552}, (2002), 77--100.
	
	\bibitem{LTV}
	B. Lou, S.P. Tan and A.D. Vo, \emph{Hyperbolic jigsaws and families of pseudomodular groups I}, Geometry \& Topology, {\bf 22}, , 2339–-2366
	
	\bibitem{Mag}
	W. Magnus, \emph{Rings of Fricke characters and automorphism groups of free groups}, Mathematische Zeitschrift {\bf 170}, (1980), 91–-103.
	
	\bibitem{Mar}
	G. A. Margulis, {\emph Discrete Subgroups of Semisimple Lie Groups}, Ergebnisse der Mathematik und ihr Grenzgebeite, Springer Berlin 1991.
	
	\bibitem{Ser}
	C. Series, {\emph The modular surface and continued fractions    }, J. London Math. Soc. (2) {\bf 31}, (1985), 69--80.
	
	\bibitem{Tak}
	K. Takeuchi,\emph{ A characterization of arithmetic Fuchsian groups}, J. Math. Soc. Japan {\bf 27}, (1975), 600–612.
	
	
%
%
%
%
%
%
%
%
%
%
%
%
%
%
%
%
%
%
%
%
%
%
%
%
%
%
%
%
%
%
	
	
\end{thebibliography}
\end{document}